\documentclass[a4paper,12pt]{article}
\usepackage[text={6.5in,8.4in},centering]{geometry}
\usepackage{graphicx,url,subfig}
\RequirePackage[OT1]{fontenc}
\RequirePackage{amsthm,amsmath,amssymb,amscd}
\RequirePackage[numbers]{natbib}
\RequirePackage[colorlinks,citecolor=blue,urlcolor=blue]{hyperref}
\usepackage{enumerate}
\usepackage{datetime}
\usepackage{etex}
\makeatother
\numberwithin{equation}{section}
\allowdisplaybreaks

\theoremstyle{plain}
\newtheorem{theorem}{Theorem}[section]

\newtheorem{lemma}[theorem]{Lemma}
\newtheorem{proposition}[theorem]{Proposition}

\theoremstyle{definition}
\newtheorem{definition}[theorem]{Definition}

\newtheorem{remark}[theorem]{Remark}

\newcommand{\E}{\mathbb{E}}

\newcommand{\ud}{\ensuremath{\mathrm{d}}}

\newcommand{\sgn}{\text{sgn}}

\newcommand{\Indt}[1]{1_{\left\{#1 \right\}}}

\newcommand{\Norm}[1]{\left|\left|  #1   \right|\right|}

\newcommand{\calB}{\mathcal{B}}

\newcommand{\calF}{\mathcal{F}}

\newcommand{\calK}{\mathcal{K}}
\newcommand{\calH}{\mathcal{H}}
\newcommand{\calI}{\mathcal{I}}
\newcommand{\calL}{\mathcal{L}}
\newcommand{\calM}{\mathcal{M}}
\newcommand{\calN}{\mathcal{N}}

\newcommand{\calP}{\mathcal{P}}

\newcommand{\bbC}{\mathbb{C}}
\newcommand{\bbN}{\mathbb{N}}

\newcommand{\R}{\mathbb{R}}

\DeclareMathOperator{\Lip}{\mathit{L}}
\DeclareMathOperator{\LIP}{Lip}
\DeclareMathOperator{\lip}{\mathit{l}}

\DeclareMathOperator{\Vip}{\overline{\varsigma}}
\DeclareMathOperator{\vip}{\underline{\varsigma}}
\DeclareMathOperator{\vv}{\varsigma}

\newcommand{\Dxa}{{}_x D_\delta^a}

\newcommand{\lMr}[3]{\:{}_{#1} #2_{#3}}

\newcommand{\myRef}[2]{#1}

\title{Moments, Intermittency and Growth Indices for\\
 the Nonlinear Fractional Stochastic Heat
Equation }

\author{{\bf Le Chen}\footnote{Research
partially supported by the Swiss National Foundation for Scientific
Research.} \footnote{Current address: Department of Mathematics, University of Utah,
155 S 1400 E RM 233, Salt Lake City, Utah, 84112-0090, USA. \textit{E-mail:} \url{chenle02@gmail.com}}
\,and
{\bf Robert C. Dalang$^*$\footnote{Address: Institut de math\'ematiques, \'Ecole Polytechnique F\'ed\'erale de Lausanne, Station 8,
CH-1015 Lausanne,
Switzerland. \textit{E-mail:} \url{robert.dalang@epfl.ch}}}\\
\\
\it\small \'Ecole Polytechnique F\'ed\'erale de Lausanne \\
\date{}
}

\begin{document}
\maketitle
\begin{center}
\begin{minipage}[rct]{5 in}
\footnotesize \textbf{Abstract:}
We study the nonlinear fractional stochastic heat equation in
the spatial domain $\R$ driven by space-time white noise.  The initial
condition is taken to be a measure on $\R$, such as the Dirac delta function,
but this measure may also have non-compact support.
Existence and uniqueness, as well as upper and lower bounds on all $p$-th
moments $(p\ge 2)$, are obtained. These bounds are uniform in the spatial variable, which answers an open problem mentioned in Conus and Khoshnevisan \cite{ConusKhosh10Farthest}.
We improve the weak intermittency statement by Foondun and Khoshnevisan
\cite{FoondunKhoshnevisan08Intermittence}, and we show that the growth indices (of linear type) introduced in
\cite{ConusKhosh10Farthest} are infinite. We introduce the notion of
``growth indices of exponential type" in order to characterize the manner in which high peaks propagate away from the origin, and
we show that the presence of a fractional differential operator leads to significantly different behavior compared with the standard stochastic heat equation.

\vspace{2ex}
\textbf{MSC 2010 subject classifications:}
Primary 60H15. Secondary 60G60, 35R60.

\vspace{2ex}
\textbf{Keywords:}
nonlinear fractional stochastic  heat equation, parabolic Anderson model,
rough initial data, intermittency, growth indices, stable processes.
\vspace{4ex}
\end{minipage}
\end{center}

\setlength{\parindent}{1em}



\section{Introduction}

In this paper, we consider the following nonlinear fractional stochastic
heat
equation:
\begin{align}\label{E:FracHt}
 \begin{cases}
  \left(\displaystyle\frac{\partial}{\partial t} - \Dxa \right) u(t,x) =
\rho\left(u(t,x)\right) \dot{W}(t,x),& t\in \R_+^*:=\;]0,+\infty[\;,\: x\in\R,\cr
u(0,\circ) = \mu(\circ),
 \end{cases}
\end{align}
where $a \in \;]0,2]$ is the order of the fractional differential
operator $\Dxa$ and $\delta$ ($|\delta|\le 2-a$) is its
skewness, $\dot{W}$ is the space-time white noise,
$\mu$ is the initial data (a measure),
the function $\rho:\R\mapsto \R$ is Lipschitz continuous,
and $\circ$ denotes the spatial dummy variable.
We refer to \cite{MainardiEtc01Fundamental} and \cite{Debbi06Explicit,
DebbiDozzi05On} for more details on these fractional differential operators.

This equation falls into a class of equations studied by Debbi and Dozzi \cite{DebbiDozzi05On}.
According to \cite[Theorem 11]{Dalang99}, even the linear form of \eqref{E:FracHt} ($\rho\equiv 1$) does not have a solution if $a\le 1$,
so they consider $a\in \:]1,2]$.
If we focus on deterministic initial conditions, then in our setting \eqref{E:FracHt}, they proved in \cite[Theorem
1]{DebbiDozzi05On} that there is a unique random field
solution if $\mu$ has a bounded density.
Equation \eqref{E:FracHt} is of particular interest since it is an extension of the classical {\it parabolic Anderson model} \cite{CarmonaMolchanov94}, in which $a=2$ and $\delta = 0$, so $\Dxa$ is the operator
$\partial^2/\partial x^2$, and $\rho(u) = \lambda u$ is a linear function.
Foondun and Khoshnevisan \cite{FoondunKhoshnevisan08Intermittence} considered problem
\eqref{E:FracHt} with the operator $\Dxa$ replaced by the $L^2\left(\R\right)$-generator $\calL$ of a L\'evy process.
They proved the existence of a random field solution under the assumption that
the initial data $\mu$ has a bounded and nonnegative density.
In \cite{ConusEct12Initial}, the operator $\Dxa$ is replaced by the generator of
a symmetric L\'evy process and the authors prove that $\mu$ can be any finite
Borel measure on $\R$.
Recently Balan and Conus \cite{BalanConus13HeatWave} studied the problem when the noise is Gaussian,
spatially homogeneous and behaves in time like a fractional Brownian motion with Hurst index $H>1/2$.

In the spirit of \cite{ChenDalang13Heat}, we begin by extending the above results (for the operator $\Dxa$) to allow a wider class of
initial data:
Let
$\calM\left(\R\right)$ be the set of signed Borel measures on $\R$.
From the Jordan
decomposition, $\mu=\mu_+-\mu_-$ where $\mu_\pm$ are two non-negative Borel
measures with disjoint support, and denote $|\mu|=\mu_++\mu_-$.
Then our admissible initial data is $\mu\in\calM_a(\R)$, where
\[
\calM_a\left(\R\right):=\left\{
\mu\in\calM(\R):\;\sup_{y\in\R}\int_\R |\mu|(\ud x) \frac{1}{1+|x-y|^{1+a}}<+\infty
\right\},\quad\text{for $a\in \:]1,2]$.}
\]
Moreover, we obtain estimates for the moments $\E(|u(t,x)|^p)$ for all $p\ge 2$.

Let us define the {\it upper and lower Lyapunov exponents of order $p$} by
\begin{align}\label{E:Lypnv-x}
\overline{m}_p(x) :=\mathop{\lim\sup}_{t\rightarrow+\infty}
\frac{1}{t}\log\E\left(|u(t,x)|^p\right),\quad
\underline{m}_p(x) :=\mathop{\lim\inf}_{t\rightarrow+\infty}
\frac{1}{t}\log\E\left(|u(t,x)|^p\right),
\end{align}
for all $p\ge 2$ and $x\in\R$.
If the initial data is constant, then $\underline{m}_p$ and $\overline{m}_p$ do
not depend on $x$. In this case, a solution is called {\it fully intermittent}
if $\underline{m}_2>0$ and $m_1=0$ by Carmona and Molchanov \cite[Definition
III.1.1, on p. 55]{CarmonaMolchanov94}. For a detailed discussion of the meaning of this intermittency property, see \cite{khosh1}. Informally, it means that the sample paths of $u(t,x)$ exhibit
``high peaks" separated by ``large valleys."

Foondun and Khoshnevisan proved {\it weak intermittency} in
\cite{FoondunKhoshnevisan08Intermittence}, namely, for all $p\ge 2$,
\[
\overline{m}_2(x) >0\;,\quad\text{and}\quad
\overline{m}_p(x) <+\infty\quad\text{for all $x\in\R$}\;,
\]
under the conditions that $\mu(\ud x)=f(x)\ud x$ with $\inf_{x\in\R} f(x)>0$ and $\inf_{x\ne 0}|\rho(x)/x|>0$.
We improve this result by showing in Theorem \ref{T:Intermittency} that
when $1<a<2$, $|\delta|<2-a$ (strict inequality) and $\mu\in\calM_a(\R)$ is nonnegative and nonvanishing,
then for all $p\ge 2$,
\begin{align*}
\inf_{x\in\R}\underline{m}_p(x)>0,\quad\text{and}\quad \sup_{x\in\R} \overline{m}_p(x)<+\infty.
\end{align*}
For this, we need a growth condition on $\rho$,
namely, that for some constants $\lip_\rho>0$ and $\vip \ge 0$,
\begin{align}\label{E:lingrow}
\rho(x)^2\ge \lip_{\rho}^2\left(\vip^2+x^2\right),\qquad \text{for
all $x\in\R$}\;.
\end{align}
In a forthcoming paper \cite{ChenKim14CP}, this weak intermittency property will be extended to full intermittency by showing
 in addition that $m_1(x)\equiv 0$.

Our result  answers an open problem
stated by Conus and Khoshnevisan \cite{ConusKhosh10Farthest}. Indeed, for the
case of the fractional Laplacian, which corresponds to our setting with $a\in \;]1,2[$
and
$\delta =0$, they ask whether the function $t\mapsto \sup_{x\in\R}
\E\left(|u(t,x)|^2\right)$ has exponential growth in $t$ for initial data with
exponential decay.
Our answer is ``yes'' under the condition \eqref{E:lingrow}.
In addition, under these conditions,
if $\mu\in\calM_{a,+}\left(\R\right)$ (where the ``$+$" sign in the subscript $\calM_{a,+}(\R)$ refers to the subset of nonnegative
measures) and $\mu\ne 0$, then for fixed $x\in\R$, the
function
$t\mapsto \E\left(|u(t,x)|^2\right)$ has at least exponential
growth; see Remark \ref{R:ExpInc}.

When the initial data are supported near the origin,
we define the following {\it growth indices of exponential type}:
\begin{align}
\label{E:SupGrowInd-0}
\underline{e}(p):= &
\sup\left\{\alpha>0: \underset{t\rightarrow \infty}{\lim}
\frac{1}{t}\sup_{|x|\ge \exp(\alpha t)} \log \E\left(|u(t,x)|^p\right) >0
\right\},\\
\label{E:SupGrowInd-1}
 \overline{e}(p) := &
\inf\left\{\alpha>0: \underset{t\rightarrow \infty}{\lim}
\frac{1}{t}\sup_{|x|\ge \exp(\alpha t)} \log \E\left(|u(t,x)|^p\right) <0
\right\}\;,
\end{align}
in order to give a proper characterization of the propagation speed of ``high peaks''.
This concept is discussed in Conus and Khoshnevisan \cite{ConusKhosh10Farthest}. These authors define
analogous indices $\underline{\lambda}(p)$ and
$\overline{\lambda}(p)$,
in which $|x|\ge \exp(\alpha t)$ is replaced by $|x|\ge \alpha t$,
which we call {\it growth indices of linear type}.

Conus and Khoshnevisan \cite{ConusKhosh10Farthest} consider the case where $\Dxa$ is replaced by the generator $\calL$ of a
real-valued symmetric L\'evy process $\left\{X_t\right\}_{t\ge 0}$. They showed
in \cite[Theorem 1.1 and Remark 1.2]{ConusKhosh10Farthest} that if the initial
data $\mu$ is a nonnegative lower semicontinuous function with certain
exponential decay at infinity, and if $X_1$ has exponential moments, then
\[
0<\underline{\lambda}(p)\le \overline{\lambda}(p)<+\infty\;,\quad
\text{for all $p\in [2,+\infty)$}\;.
\]
An important example of such a L\'evy process is the
``truncated symmetric stable process".

Here, we will  be able to consider (not necessarily symmetric) stable processes with $a\in\;]1,2]$,
for which even the second moment of $X_1$ does not exist, and we will see that when $1<a<2$, the presence of the fractional differential operator $\Dxa$ leads to significantly different behaviors of
the speed of propagation of high peaks, compared to that obtained in \cite{ConusKhosh10Farthest}.

   First, we show that if the initial data has sufficient decay at $\pm \infty$, then $\overline{e}(p) < \infty$. Then we show that if $1<a<2$ and $\vert\delta\vert < 2-a$ (meaning that the underlying
stable process has both positive and negative jumps), then
\begin{align}\label{E:ep}
\underline{e}(p)>0\;,\quad
\text{for all $p\in [2,+\infty)$ and $\mu\in
\calM_{a,+}(\R)$, $\mu\ne
0$}\:,
\end{align}
provided $\rho$ satifies condition \eqref{E:lingrow}. This conclusion applies, for instance, to the case where the initial data $\mu$ is the Dirac delta function. In particular, for well-localized
initial data (for instance, $\mu$ has a positive moment), $0< \underline{e}(p) \leq \overline{e}(p) < +\infty$, whereas for initial data that is bounded below ($\mu(dx) = f(x) dx$ with $f(x) > c >0$,
for all $x\in \R$), $\underline{e}(p) = \overline{e}(p) = +\infty$. See Theorem \ref{T:Growth} for the precise statements.
As a direct consequence, $\underline{\lambda}(p)=
\overline{\lambda}(p)=+\infty$ for all $p\in [2,\infty[$.



The structure of this paper is as follows. After introducing some preliminaries in Section \ref{S5:Pre}, the main results are
presented in Section \ref{S5:MainRes}: Existence and general bounds are given  in Theorem \ref{T:ExUni}, followed by explicit upper and lower
bounds on the function $\calK$. These lead to our results on  weak intermittency (Theorem \ref{T:Intermittency})
and growth indices (Theorem \ref{T:Growth}). Section \ref{SS:ExUni} contains the proof of Theorem \ref{T:ExUni} and Section  \ref{S:IntGrow}
presents the proofs of Theorems \ref{T:Intermittency} and \ref{T:Growth}.

\section{Some preliminaries and notation}\label{S5:Pre}
The Green function associated to the
problem
\eqref{E:FracHt} is
\begin{align}\label{E:Green}
\lMr{\delta}{G}{a}(t,x) := \calF^{-1}
\left[\exp\left\{\lMr{\delta}{\psi}{a}(\cdot) t\right\} \right](x)
= \frac{1}{2\pi}\int_\R \ud \xi\: \exp\left\{i \xi x - t |\xi|^a
e^{-i \delta \pi\: \sgn(\xi)/2}\right\},
\end{align}
where $\calF^{-1}$ is the inverse Fourier transform and
\[
\lMr{\delta}{\psi}{a}(\xi) = -|\xi|^a e^{-i \delta \pi\:
\sgn(\xi)/2}\;.
\]
Denote the solution to the homogeneous equation
\begin{align*}
 \begin{cases}
  \left(\displaystyle\frac{\partial}{\partial t} - \Dxa \right) u(t,x) = 0,&
t\in \R_+^*\;,\: x\in\R,\cr
u(0,\circ) = \mu(\circ),
 \end{cases}
\end{align*}
by
\[
J_0(t,x) := \left(\lMr{\delta}{G}{a}(t,\circ)*\mu\right)(x)
=\int_\R \mu(\ud y)\:\lMr{\delta}{G}{a}(t,x-y),
\]
where ``$*$'' denotes the convolution in the space variable.

Let $W=\left\{
W_t(A),\, A\in\calB_b(\R),\, t\ge 0 \right\}$
be a space-time white noise
defined on a probability space $(\Omega,\calF,P)$, where
$\calB_b\left(\R\right)$ is the
collection of Borel sets with finite Lebesgue measure.
Let  $(\calF_t,\, t\ge 0)$ be the filtration generated by $W$ and augmented by the $\sigma$-field $\calN$ generated
by all $P$-null sets in $\calF$:
\[
\calF_t = \sigma\left(W_s(A):0\le s\le
t,A\in\calB_b\left(\R\right)\right)\vee
\calN,\quad t\ge 0.
\]
In the following, we fix this filtered
probability space $\left\{\Omega,\calF,\{\calF_t:t\ge0\},P\right\}$.
We use $\Norm{\cdot}_p$ to denote the
$L^p(\Omega)$-norm ($p\ge 1$).
With this setup, $W$ becomes a worthy martingale measure in the sense of Walsh
\cite{Walsh86}, and $\iint_{[0,t]\times\R}X(s,y) W(\ud s,\ud y)$ is
well-defined in this reference for a suitable class of random fields
$\left\{X(s,y),\; (s,y)\in\R_+\times\R\right\}$.

The rigorous meaning of the spde \eqref{E:FracHt} uses the integral formulation
\begin{equation}
\label{E:WalshSI}
 \begin{aligned}
  u(t,x) &= J_0(t,x)+I(t,x),\quad\text{where}\cr
I(t,x) &=\iint_{[0,t]\times\R}
\lMr{\delta}{G}{a}\left(t-s,x-y\right)\rho\left(u(s,x)\right)W(\ud s,\ud
y).
 \end{aligned}
\end{equation}

\begin{definition}\label{DF:Solution}
A process $u=\left(u(t,x),\:(t,x)\in\R_+^*\times\R\right)$  is called a {\it
random field solution} to
\eqref{E:FracHt} if:
\begin{enumerate}[(1)]
 \item $u$ is adapted, i.e., for all $(t,x)\in\R_+^*\times\R$, $u(t,x)$ is
$\calF_t$-measurable;
\item $u$ is jointly measurable with respect to
$\calB\left(\R_+^*\times\R\right)\times\calF$;
\item $\left(\lMr{\delta}{G}{a}^2 \star \Norm{\rho(u)}_2^2\right)(t,x)<+\infty$
for all $(t,x)\in\R_+^*\times\R$,
where ``$\star$'' denotes the simultaneous
convolution in both space and time
variables;
\item For all $(t,x)\in\R_+^*\times\R$, $u(t,x)$ satisfies \eqref{E:WalshSI} a.s.;
\item The function $(t,x)\mapsto I(t,x)$ mapping $\R_+^*\times\R$ into
$L^2(\Omega)$ is continuous;
\end{enumerate}
\end{definition}

Assume that the function $\rho:\R\mapsto \R$ is globally Lipschitz
continuous with Lipschitz constant $\LIP_\rho>0$.
We need some growth conditions on $\rho$:
Assume that for some constants $\Lip_\rho>0$ and $\Vip \ge 0$,
\begin{align}\label{E:LinGrow}
\rho(x)^2 \le \Lip_\rho^2 \left(\Vip^2 +x^2\right),\qquad \text{for
all $x\in\R$}.
\end{align}
Note that $\Lip_\rho\le \sqrt{2}\LIP_\rho$, and the inequality may be strict.
We shall also specially consider the linear case: $\rho(u)=\lambda u$ with
$\lambda\ne 0$, which is related to the
{\it parabolic Anderson model} ($a=2$). It is a
special case of the following near-linear growth condition:
for some constant $\vv\ge 0$,
\begin{align}\label{E:qlinear}
\rho(x)^2= \lambda^2\left(\vv^2+x^2\right),\qquad  \text{for
all $x\in\R$}\;.
\end{align}

For all $(t,x)\in\R_+^*\times \R$, $n\in\bbN$ and $\lambda\in\R$, define
\begin{align}
\notag
\calL_0\left(t,x;\lambda\right) &:= \lambda^2 \lMr{\delta}{G}{a}^2(t,x),  \\
\label{E:Ln}
\calL_n\left(t,x;\lambda\right)&:=
\underbrace{\left(\calL_0\star \cdots\star\calL_0\right)}_{\text{$n+1$
factors } \calL_0(\cdot,\circ;\lambda)}
,\quad\text{for $n\ge 1$,}
\end{align}
and
\begin{align}
\calK\left(t,x;\lambda\right)&:= \sum_{n=0}^\infty
\calL_n\left(t,x;\lambda\right),
\label{E:K}
\end{align}
(the convergence of this series is established in Proposition \ref{P:UpperBdd-K}).
For $t\ge 0$, define
\[
\calH(t;\lambda) := \left(1\star \calK(\cdot,\circ;\lambda)\right)(t,x).
\]

Let $z_p$ be the the universal constant in the Burkholder-Davis-Gundy
inequality (in particular, $z_2=1$), and so $z_p\le 2\sqrt{p}$ for all $p\ge 2$; see \cite[Appendix]{CarlenKree91}. Define
\[
a_{p,\Vip} =
\begin{cases}
2^{(p-1)/p} & \text{if $\Vip\ne 0$ and $p>2$},\cr
\sqrt{2} & \text{if $\Vip= 0$ and $p>2$},\cr
1 & \text{if $p=2$.}
\end{cases}
\]
We apply the following conventions to the
kernel functions $\calK(t,x;\lambda)$ (and similarly to $\calH(t;\lambda)$):
\begin{align*}
 \calK(t,x) &:= \calK(t,x;\lambda),&
\overline{\calK}(t,x) &:= \calK\left(t,x;\Lip_\rho\right),\\
\underline{\calK}(t,x) &:= \calK\left(t,x;\lip_\rho\right),&
\widehat{\calK}_p(t,x) &:= \calK\left(t,x;a_{p,\Vip}\:
z_p \Lip_\rho\right),\quad\text{for $p\ge 2$}\:.
\end{align*}

\section{Main results}\label{S5:MainRes}

\subsection{Existence, uniqueness and moments}
The following theorem extends the result of \cite[Theorem 2.4]{ChenDalang13Heat} from $a=2$ to $a\in \:]1,2]$.
In view of the related result \cite[Theorem 2.3]{ChenDalang13Wave} and Remark 2.4 in this reference,
the bounds in this theorem are not a surprise, though they do require a proof. The main effort will be to turn
these abstract bounds into concrete estimates, via explicit upper and lower bounds on the functions $\calK$ and $\calH$ (see Section \ref{sec3.2}).
For $\tau\ge t> 0$ and $x, y\in\R$, define
\begin{align*}
\calI(t,x,\tau, y;\vv,\lambda) :=&
\lambda^2\int_0^t \ud r\int_\R  \ud z
\left[
J_0^2(r,z)+\left(J_0^2(\cdot,\circ)\star
\calK(\cdot,\circ;\lambda)\right)(r,y)+ \vv^2
\left(\calH(r;\lambda)+1\right)
\right]\\
& \qquad\qquad \times \lMr{\delta}{G}{a}(t-r,x-z)\lMr{\delta}{G}{a}(\tau-r,y-z).
\end{align*}

\begin{theorem}[Existence,uniqueness and moments]\label{T:ExUni}
Suppose that
\begin{enumerate}[(i)]
\item $1<a\le 2$ and $|\delta|\le 2-a$;
 \item the function $\rho$ is Lipschitz continuous and satisfies the growth
condition \eqref{E:LinGrow};
 \item the initial data are such that $\mu\in \calM_a\left(\R\right)$.
\end{enumerate}
Then the stochastic pde \eqref{E:FracHt} has a random
field solution $\{u(t,x)\!: (t,x)\in\R_+^* \times \R \}$. Moreover:\\
(1) $u(t,x)$ is unique in the sense of versions;\\
(2) $(t,x)\mapsto u(t,x)$ is $L^p(\Omega)$-continuous for all integers $p\ge
2$; \\
(3) For all even integers $p\ge 2$, all $\tau\ge t>0$ and $x,y\in\R$,
\begin{align}\label{E:SecMom-Up}
 \Norm{u(t,x)}_p^2 \le
\begin{cases}
 J_0^2(t,x) + \left(\left[\Vip^2+J_0^2\right] \star \overline{\calK} \right)
(t,x),& \text{if $p=2$}\;,\cr
2J_0^2(t,x) + \left(\left[\Vip^2+2J_0^2\right] \star \widehat{\calK}_p \right)
(t,x),& \text{if $p>2$}\;,
\end{cases}
\end{align}
and
\begin{equation}
\label{E:TP-Up}
 \E\left[u(t,x)u\left(\tau,y\right)\right] \le
J_0(t,x)J_0\left(\tau,y\right) + \calI(t,x,\tau,y;\Vip,\Lip_\rho)\;.
\end{equation}
(4) If $\rho$ satisfies \eqref{E:lingrow}, then for  all $\tau\ge t>0$ and $x,y\in\R$,
\begin{align}
\label{E:SecMom-Lower}
\Norm{u(t,x)}_2^2 \ge J_0^2(t,x) + \left(\left(\vip^2+J_0^2\right) \star
\underline{\calK} \right)
(t,x),
\end{align}
and
\begin{equation}
\label{E:TP-Lower}
\E\left[u(t,x)u\left(\tau,y\right)\right] \ge
J_0(t,x)J_0\left(\tau,y\right)  + \calI(t,x,\tau,y;\vip,\lip_\rho)\;.
\end{equation}
(5) If $\rho$ satisfies \eqref{E:qlinear}, then for  all $\tau\ge t>0$ and $x,y\in\R$,
\begin{align}
\label{E:SecMom-Lin}
\Norm{u(t,x)}_2^2 = J_0^2(t,x) + \left(\left(\vv^2+J_0^2\right) \star
\calK \right)
(t,x),
\end{align}
and
\begin{equation}
\label{E:TP-Lin}
\E\left[u(t,x)u\left(\tau,y\right)\right] =
J_0(t,x)J_0\left(\tau,y\right)  + \calI(t,x,\tau,y;\vv,\lambda)\;.
\end{equation}
\end{theorem}
The proof of this theorem is given in Section \ref{SS:ExUni}.

\subsection{Estimates on the kernel function $\calK(t,x)$}\label{sec3.2}
Recall that if the partial differential operator is the heat operator
$\frac{\partial }{\partial  t}-\frac{\nu}{2} \Delta$, then
\begin{align}\label{E:K-Heat}
\calK^{\mbox{\scriptsize heat}}(t,x;\lambda) = G_{\frac{\nu}{2}}(t,x)
\left(\frac{\lambda^2}{\sqrt{4\pi\nu t}}+\frac{\lambda^4}{2\nu}
\: e^{\frac{\lambda^4 t}{4\nu}}\Phi\left(\lambda^2
\sqrt{\frac{t}{2\nu}}\right)\right),
\end{align}
where $\nu>0$ and $\Phi(x)$ is the distribution function of the standard norm
random variable; see \cite[Proposition 2.2]{ChenDalang13Heat}. When the partial
differential operator is the wave operator
$\frac{\partial ^2 }{\partial t^2}-\kappa^2 \Delta$,
\begin{align}
\label{E:K-Wave}
\calK^{\mbox{\scriptsize wave}}(t,x;\lambda) = \frac{\lambda^2}{4}
I_0\left(\sqrt{\frac{\lambda^2((\kappa t)^2-x^2)}{2\kappa}}\right) \Indt{|x|\le
\kappa t},
\end{align}
where $\kappa>0$ and $I_0(x)$ is the modified Bessel function of the first kind
of order $0$; see \cite[Proposition 3.1]{ChenDalang13Wave}.

Except in the above two cases, we do not have an
explicit formula for the kernel function $\calK(t,x)$ in \eqref{E:K}.
In order to make use of the moment formulas in \eqref{E:SecMom-Up} and
\eqref{E:SecMom-Lower}, we derive upper and lower bounds on this
kernel function
in the following two propositions.
We will need the two-parameter {\it Mittag-Leffler
function} \cite[Section 1.2]{Podlubny99FDE}:
\begin{align}\label{E:Mittag-Leffler}
E_{\alpha,\beta}(z) := \sum_{k=0}^{\infty} \frac{z^k}{\Gamma(\alpha k+\beta)},
\qquad \alpha>0,\;\beta> 0\;,
\end{align}
where $\Gamma(x) = \int_0 ^\infty e^{-t} t^{x-1} \ud t $ is Euler's Gamma function \cite{NIST2010}.
Let $a^*$ be the dual of $a$: $1/a+1/a^*=1$.
By Lemma \ref{L:Green} below (Property (ii)),  the following constant
\begin{align}\label{E:Cst-dLa}
\Lambda=\lMr{\delta}{\Lambda}{a} := \sup_{x\in\R} \lMr{\delta}{G}{a}(1,x)\;
\end{align}
is finite. In particular,
\[
\lMr{0}{\Lambda}{a} = \lMr{0}{G}{a}(1,0) =
\frac{1}{2\pi} \int_\R\ud \xi\: \exp\left(-|\xi|^a \right) =
\frac{1}{a \pi} \int_0^\infty\ud t\: e^{-t} t^{1/a -1}
=
\frac{\Gamma\left(1+1/a\right)}{\pi}\;.
\]
In the following, we often omit the dependence of this constant on $\delta$ and
$a$ and simply write $\Lambda$ instead of $\lMr{\delta}{\Lambda}{a}$.
Define
\begin{align}\label{E:D-gamma}
\begin{aligned}
\gamma&:=\lambda^2\Lambda\:\Gamma(1/a^*),&
\overline{\gamma}&:=\Lip_\rho^2\Lambda\:\Gamma(1/a^*),\\
\underline{\gamma}&:=\lip_\rho^2\Lambda\:\Gamma(1/a^*), &
\widehat{\gamma}_p&:=a_{p,\Vip}^2
z_p^2\Lip_\rho^2\Lambda\:\Gamma(1/a^*),\quad\text{for $p\ge 2$.}
\end{aligned}
\end{align}
Clearly, $\widehat{\gamma}_2 = \overline{\gamma}$.

\begin{proposition}[Upper bound on $\calK(t,x)$]
\label{P:UpperBdd-K}
Suppose that $a\in\:]1,2]$ and $|\delta|\le 2-a$.
The kernel function $\calK(t,x)$ defined in \eqref{E:K} satisfies,
for all $t\ge 0$ and $x\in\R$,
\begin{align}\label{E:UpBd-K-Mit}
\calK(t,x)& \le \lMr{\delta}{G}{a}(t,x) \frac{\gamma}{t^{1/a}}E_{1/a^*,1/a^*}
\left(
\gamma t^{1/a^*}
\right)\\
&\le
\frac{C }{t^{1/a}} \lMr{\delta}{G}{a}(t,x)
\left(
1+t^{1/a} \exp\left(\gamma^{a^*} t\right)
\right),
\label{E:UpBd-K}
\end{align}
where the constant $C=C(a,\delta,\lambda)$ can be chosen as
\begin{align}\label{E:Cst-UpK}
C(a,\delta,\lambda):= \gamma \: \sup_{t\ge 0}
\frac{E_{1/a^*,1/a^*}\left(\gamma\:
t^{1/a^*}\right)}{1+t^{1/a}\exp\left(\gamma^{a^*}
t\right)}<+\infty\;.
\end{align}
\end{proposition}
This proposition is proved in Section \ref{SS:ExUni}.
For a lower bound on $\calK(t,x)$, we need another family of kernel functions:
\begin{align}\label{E:ga}
g_a(t,x):=  \frac{1}{\pi}\: \frac{t}{\left(t^{2/a}+x^2
\right)^{\frac{a}{2}+\frac{1}{2}}}\:,\qquad \text{with $a>0$}\;.
\end{align}
These functions have the same scaling property as
$\lMr{\delta}{G}{a}(t,x)$:
\[
g_a(t,x) = \frac{1}{t^{1/a}} g_a\left(1,\frac{x}{t^{1/a}}\right).
\]
Note that $g_1(t,x)$ is nothing but the Poisson kernel (see, e.g.,
\cite[p. 268]{Yosida95FA}), which satisfies the
semigroup property
\[
\left(g_1(t-s,\cdot)*g_1(s,\cdot)\right)(x) = g_1(t,x),\qquad 0\le s\le t,\:
x\in\R.
\]
For $a\in \;]1,2[$ and $|\delta|< 2-a$, define
\begin{align}\label{E:Cad}
\widetilde{C}_{a,\delta} := \inf_{(t,x)\in\R_+^*\times\R}
\frac{\lMr{\delta}{G}{a}(t,x)}{\pi g_a(t,x)}>0\:,
\end{align}
which is strictly positive by Lemma \ref{L:Cad} below.  Then let
\begin{align}\label{E:Upsilon}
\Upsilon(\lambda,a,\delta):=
\frac{\lambda^4\: \widetilde{C}^4_{a,\delta} \, C_{a+1/2}^2}{2}\:
\Gamma\left(1-\frac{1}{a}\right)^2,
\end{align}
where
\begin{align}\label{E:Cnu}
C_\nu:= \frac{\Gamma(\nu)\Gamma(1/2)}{2\:\Gamma(\nu+1/2)}\;,\quad
\nu\ge 1/2\:.
\end{align}

\begin{proposition}[Lower bound on $\calK(t,x)$]
\label{P:LowBdd-K}
Fix $a\in \;]1,2[$ and $|\delta|< 2-a$ (note the strict inequality).
Set
\[
b=2-2/a \in \;]0,1]\,.
\]
Then
\begin{align*}
\calK(t,x)&\ge C\; t^{b-1} \: g_1\left(t^{1/a},x\right)\:
E_{b,b}\left(\Upsilon(\lambda,a,\delta) \: t^b \right),\quad
\text{for all $t>0$ and $x\in\R$},
\end{align*}
where the constant $C=C(\lambda,a,\delta)$ can be chosen to be
\[
C= 2^{-1/2}\lambda^4\:
\widetilde{C}^4_{a,\delta} \:C_{a+1/2}^2\: \Gamma(1-1/a)^2.
\]
In particular, for the same constant $C$, for all $t>0$ and $x\in\R$,
\[
\left(1\star \calK\right)(t,x) \ge
C \:t^b E_{b,b+1}\left(\Upsilon(\lambda,a,\delta) \: t^b \right).
\]
\end{proposition}

This proposition is proved in Section \ref{SS:LowBd-K}.

\subsection{Growth indices and weak intermittency}



\begin{theorem}[Weak intermittency]\label{T:Intermittency}
Suppose that $a\in \:]1,2]$ and $|\delta|\le 2-a$.\\
(1) If $\rho(u)$ satisfies \eqref{E:LinGrow} and
$\mu\in\calM_a(\R)$, then for all even integers $p\ge 2$,
\begin{align}\label{E:InterU}
 \sup_{x\in\R}\overline{m}_p(x) \le \frac{1}{2}\left(16\Lip_\rho^2
\Lambda\Gamma(1-1/a)\right)^{a/(a-1)} \: p^{2+1/(a-1)}.
\end{align}
(2) Suppose $\rho$ satisfies \eqref{E:lingrow}, $|\delta|<2-a$ (strict inequality) and $\mu\in \calM_{a,+}\left(\R\right)$.
If either $\mu\ne 0$ or $\vip\ne 0$, then for all $p\ge 2$, setting $b=2-2/a$,
\[
\inf_{x\in\R}\: \underline{m}_p(x)\ge
\frac{p}{2}
\Upsilon\left(\lip_\rho,a,\delta\right)^{1/b}>0.
\]
\end{theorem}

Note that if $a=2$,  then
for some constant $C$, we have that
$\overline{m}_p\le C p^3$, which recovers the previous analysis (see \cite{BertiniCancrini94Intermittence}, \cite[Example 2.7]{ChenDalang13Heat}, etc).


\begin{remark}\label{R:ExpInc}
Fix $p\ge 2$. Clearly, Theorem \ref{T:Intermittency} implies that for all $x\in\R$,
\begin{align*}
\mathop{\lim\inf}_{t\rightarrow\infty}
\frac{1}{t}\sup_{y\in\R} \log\E\left(|u(t,y)|^p\right)
&\ge
\mathop{\lim\inf}_{t\rightarrow\infty}
\frac{1}{t}\log\E\left(|u(t,x)|^p\right)=
\underline{m}_p(x)\ge \frac{p}{2}
\Upsilon\left(\lip_\rho,a,\delta\right)^{1/b}>0\:.
\end{align*}
Hence, the function $t\mapsto \sup_{y\in\R}\E\left(|u(t,y)|^p\right)$ has at
least exponential growth.
This answers the second open problem stated by Conus and Khoshnevisan in
\cite{ConusKhosh10Farthest}.
Moreover, Theorem \ref{T:Intermittency} implies that for all fixed $x\in\R$, the
function $t\mapsto \E\left(|u(t,x)|^p\right)$ also has at least exponential growth.
\end{remark}


Recall the definitions of the constants $\widehat{\gamma}_p$ and $\Upsilon(\lip_\rho,a,\delta)$ in \eqref{E:D-gamma} and \eqref{E:Upsilon}, respectively.

\begin{theorem}[Growth indices]\label{T:Growth}
(1) Suppose that $a\in \;]1,2]$, $|\delta|\le 2-a$ and $\rho$ satisfies \eqref{E:LinGrow}.
If there are $C < \infty$, $\alpha>0$ and $\beta > 0$ such that for all $(t,x) \in [1,\infty[\, \times \R$,
\begin{equation}\label{rd3.20}
   \vert J_0(t,x)\vert \leq C (1+ t^\alpha) (1+\vert x \vert)^{-\beta}.
\end{equation}
Then
\begin{equation}\label{rd3.20a}
\overline{e}(p)\le \frac{\widehat{\gamma}_p^{a/(a-1)}}{\beta}<+\infty.
\end{equation}
In particular, if, for some $\eta >0$,  $\int_\R|\mu|(\ud y)(1+|y|^{\eta})<\infty$, then \eqref{rd3.20} and \eqref{rd3.20a} are satisfied with
$\beta=\min(\eta, 1+a)$.
\\
\noindent
(2) Suppose that $a\in \;]1,2[$ (note that $a\ne 2$), $|\delta|< 2-a$ (strict inequality) and $\rho$ satisfies \eqref{E:lingrow}. Set $b = 2 - 2/a$.
For all $\mu\in \calM_{a,+}\left(\R\right)$, $\mu\ne
0$ and all $p\ge 2$, if $\vip =0$, then
\[
\underline{e}(p)\ge
\frac{\Upsilon\left(\lip_\rho,a,\delta\right)^{1/b}}{2}>0.
\]
For these $\mu$, if $\vip = 0$ and there is $c>0$ such that
\begin{equation}\label{rd3.20b}
   J_0(t,x) \geq c\qquad  \mbox{for all } (t,x) \in \R_+ \times \R,
\end{equation}
or if $\vip\ne 0$, then
$\underline{e}(p) = \overline{e}(p) =+\infty$.
In particular, $\underline{\lambda}(p) =
\overline{\lambda}(p)=+\infty$ for all $p\ge 2$, and a sufficient condition for \eqref{rd3.20b} is that $\mu(dx) = f(x) dx$ with $f(x) \geq c$, for all $x\in \R$.\\
\end{theorem}

The above two theorems are proved in Section \ref{S:IntGrow}.

\begin{remark} In the case of the classical parabolic Anderson model, in which $a=2$, $\delta =0$ and $\rho(u) = \lambda u$, it was shown in \cite{ChenDalang13Heat} that $\underline{\lambda}(2)
=\overline{\lambda}(2) = \lambda^2/2$ when the initial data has compact support (for instance). Here, it is natural to ask whether $\underline{e}(p) = \overline{e}(p)$ when $\rho(u) = \lambda u$, for
instance for initial data with compact support. This remains an open question.
\end{remark}

\section{Proof of Theorem \ref{T:ExUni}}\label{SS:ExUni}

We need some technical results. The proof of Theorem \ref{T:ExUni} will be presented at the end of this section.

The Green functions defined in \eqref{E:Green} are densities of stable random variables.
Some key properties are stated in the next lemma.  Recall that
a probability density function $f:\R\mapsto\R_+$ is called {\it bell-shaped} if  $f$ is infinitely differentiable and its $k$-th derivative
$f^{(k)}$ has exactly $k$ zeros in its support for all $k$.

\begin{lemma}\label{L:Green}
 For $a\in \;]0,2]$, the following properties hold:
\begin{enumerate}[(i)]
 \item For fixed $t>0$, the function $\lMr{\delta}{G}{a}(t,\cdot)$ is a
bell-shaped density function. In particular, $\int_\R \lMr{\delta}{G}{a}(t,x)
\ud x =1$.
 \item The unique mode is located on the positive semi-axis $x>0$ if  $\delta>0$
and on the negative semi-axis $x<0$ if $\delta<0$ and at $x=0$ if $\delta
=0$.
 \item $\lMr{\delta}{G}{a}(t,x)$ satisfies the semigroup property, i.e., for
$0<s<t$,
\[
\lMr{\delta}{G}{a}(t+s,x) = \int_\R \ud \xi\: \lMr{\delta}{G}{a}(t,\xi)
\lMr{\delta}{G}{a}(s,x-\xi)\:.
\]
\item The scaling property: For all $n\ge 0$,
\begin{align}\label{E:Ga-Scale}
\frac{\partial^n }{\partial x^n} \lMr{\delta}{G}{a}(t,x) = t^{-\frac{n+1}{a}}
\left. \frac{\partial^n }{\partial \xi^n} \lMr{\delta}{G}{a}(1,\xi) \right|_{\xi=
t^{-1/a} x}.
\end{align}
\item When $x\rightarrow\pm \infty$,
\begin{align*}
\lMr{\delta}{G}{a}(1,x) = \frac{1}{\pi}\sum_{j=1}^N |x|^{-a j -1}
\frac{(-1)^{j+1}}{j!}\Gamma(aj+1)\sin\left(j(a\pm\delta)\pi/2\right)
+ O\left(|x|^{-a(N+1)-1}\right).
\end{align*}
\item If $a\in \;]1,2]$, then there exists some finite constants
$K_{a,n}$ such that
\begin{align}
 \left|\lMr{\delta}{G}{a}^{(n)}(1,x)\right|
\label{E:G-bd}
&\le \frac{K_{a,n}}{1+|x|^{1+n+a}},\quad\text{for $n\ge 0$;}
\end{align}
Moreover, for all $T\ge t>0$, $n\ge 0$ and $x\in\R$,
\begin{align}\label{E:Gtx-bd}
 \left|\frac{\partial^n}{\partial x^n} \lMr{\delta}{G}{a}(t,x)\right|\le
t^{-\frac{n+1}{a}} \frac{K_{a,n}}{1+|t^{-1/a}x|^{1+n+a}}
\le  K_{a,n}\: t^{-\frac{n+1}{a}} \frac{(T\vee 1)^{1+\frac{n+1}{a}}}{1+|x|^{1+n+a}}\:.
\end{align}
\item $\lim_{t\rightarrow 0} \lMr{\delta}{G}{a}(t,x) = \delta_0(x)$, where
$\delta_0(x)$ is the Dirac delta function with unit mass at zero.
\end{enumerate}
\end{lemma}
\begin{proof}
Most of these properties appear in several books \cite{Zolotarev86,
UchaikinZolotarev99, Lukacs70}. We refer the
interested readers to \cite[Lemma 1]{Debbi06Explicit} for Properties (i) (except
the bell-shaped density), (iii) and (iv). Formula (v) can be find in \cite[(5.9.3), Sec. 5.9]{Lukacs70}.
The proof that the density is bell-shaped
is due to Gawronski \cite{Gawronski84Bell}.
Property (ii) can be found in the summary part of \cite[Section
2.7, p. 143--147]{Zolotarev86}.

Now we prove (vi).
Property \eqref{E:G-bd} follows from \cite[Corollary 1]{DebbiDozzi05On}.
By the scaling property \eqref{E:Ga-Scale} and \eqref{E:G-bd},
\begin{align*}
\left|\frac{\partial^n }{\partial x^n} \lMr{\delta}{G}{a}(t,x)\right| &\le t^{-\frac{n+1}{a}} \frac{K_{a,n}}{1+|t^{-1/a}x|^{1+n+a}} = t^{-\frac{n+1}{a}}\frac{K_{a,n} \:
t^{1+\frac{n+1}{a}}}{t^{1+\frac{n+1}{a}}+|x|^{1+n+a}}.
\end{align*}
Then using the fact that the function $t\mapsto \frac{t}{t+z}$ is monotone increasing on $\R_+$, the above quantity is less than
\begin{align*}
t^{-\frac{n+1}{a}}\frac{K_{a,n} \: (T\vee 1)^{1+\frac{n+1}{a}}}{(T\vee 1)^{1+\frac{n+1}{a}}+|x|^{1+n+a}}
\le  t^{-\frac{n+1}{a}}\frac{K_{a,n} \: (T\vee 1)^{1+\frac{n+1}{a}}}{1+|x|^{1+n+a}}\:.
\end{align*}
This proves \eqref{E:Gtx-bd}.

Property (vii) follows easily by taking Fourier transforms $\calF(\lMr{\delta}{G}{a}(t,\cdot))(\xi) =
\exp\left(\lMr{\delta}{\psi}{a}(\xi)t \right) \rightarrow 1$ as $t\rightarrow
0_+$.
This completes the proof of Lemma \ref{L:Green}.
\end{proof}
%


Let $\calL_n(t,x;\lambda)$ and $\calK(t,x;\lambda)$, and
$\Lambda=\lMr{\delta}{\Lambda}{a}$ be defined in \eqref{E:Ln},
\eqref{E:K}, and \eqref{E:Cst-dLa}, respectively.
Recall that $1/a+1/a^*=1$.

\begin{lemma}[Theorem 1.3, p. 32 in \cite{Podlubny99FDE}]\label{L:Eab}
If $0<\alpha<2$, $\beta$ is an arbitrary complex number and $\mu$ is an
arbitrary real number such that
\[
\pi\alpha/2<\mu<\pi \wedge (\pi\alpha)\;,
\]
then for an arbitrary integer $p\ge 1$ the following expression holds:
\[
E_{\alpha,\beta}(z) = \frac{1}{\alpha} z^{(1-\beta)/\alpha}
\exp\left(z^{1/\alpha}\right)
-\sum_{k=1}^p \frac{z^{-k}}{\Gamma(\beta-\alpha k)} + O\left(|z|^{-1-p}\right)
,\quad |z|\rightarrow\infty,\quad |\arg(z)|\le \mu\:.
\]
\end{lemma}

\begin{proposition}\label{P:K}
 For $1<a\le 2$, $|\delta|\le 2-a$ and $\lambda>0$, we have the following
properties:
\begin{enumerate}[(i)]
 \item $\calL_n(t,x;\lambda)$ is non-negative and for all $n\ge 0$ and $(t,x)\in \R_+^*\times\R$,
\begin{align}\label{E:K-i-A}
 \calL_n(t,x;\lambda) \le \:B_{n+1}(t;\lambda) \lMr{\delta}{G}{a}(t,x)
\;,
\end{align}
where
\[
B_n\left(t;\lambda\right):=
\lambda^{2n} \Lambda^{n}
\frac{\Gamma\left(1/a^*\right)^{n}}
{\Gamma\left(n/a^*\right) }\: t^{n/a^*-1}\qquad
(n\ge 0,\;\lambda\in\R).
\]
\item For all $t>0$ and $\lambda>0$, the series $\sum_{n=1}^\infty
\calL_n(t,x;\lambda) $ converges uniformly over
$x\in\R$ and hence $\calK(t,x;\lambda)$ in \eqref{E:K} is well defined.
\item $B_n\left(t\;;\lambda\right)\ge 0$ and for all
$m\in\bbN^*$, $\sum_{n=0}^\infty
B_n\left(t\;;\lambda\right)^{1/m}<+\infty$.
\end{enumerate}
\end{proposition}
\begin{proof}

(i) Non-negativity is clear. The scaling property \eqref{E:Ga-Scale} and the definition of $\Lambda$ in
\eqref{E:Cst-dLa} imply that
\begin{align}\label{E5_:Ga0bd-A}
\lMr{\delta}{G}{a}(t,x)\le t^{-1/a} \Lambda\;,
\end{align}
which establishes the case $n=0$ in \eqref{E:K-i-A}.
Suppose that the relation \eqref{E:K-i-A} holds up to $n-1$. Then by
\eqref{E5_:Ga0bd-A}, we have
\begin{align*}
\calL_n(t,x;\lambda) =& \int_0^t \ud s \int_\R \ud y\:
\calL_{n-1}\left(t-s,x-y\right)
\lambda^2
\lMr{\delta}{G}{a}^2\left(s,y\right)\\
\le&
\lambda^{2(n+1)} \Lambda^{n+1}
\frac{\Gamma\left(1/a^*\right)^n}{\Gamma\left(n(1/a^*)\right)}
\int_0^t\ud s\: (t-s)^{n(1/a^*)-1}  s^{-1/a}\\
&\qquad\times \int_\R\ud y\:
\lMr{\delta}{G}{a}\left(t-s,x-y\right)
\lMr{\delta}{G}{a}\left(s,y\right).
\end{align*}
The conclusion now follows from the semigroup property of
$\lMr{\delta}{G}{a}(t,x)$ and Euler's Beta integral (see \cite[5.12.1, on p.
142]{NIST2010})
\begin{align} \label{E:BI}
\int_0^t\ud s \: s^{a-1}(t-s)^{b-1}=
\frac{\Gamma(a)\Gamma(b)}{\Gamma(a+b)}\, t^{a+b-1},\qquad \text{with $\Re(a)>0$ and
$\Re(b)>0$.}
\end{align}

(ii) This is a consequence of (iii). As for (iii), the non-negativity is clear.
By \eqref{E:K-i-A} and \eqref{E5_:Ga0bd-A},
\[
\calL_n(t,x;\lambda)\le B_{n+1}\left(t\;;\lambda\right)
t^{-1/a}\Lambda\;.
\]
Thus, if the series  $\sum_n B_n\left(t\;;\lambda\right)^{1/m}$
converges, then $\calL_n$ does so uniformly over $x\in\R$.
Denote $\beta :=1/a^*$.
We use the ratio test:
\[
\left(\frac{B_n\left(t\;;\lambda\right)}{B_{n-1}\left(t\;;\lambda\right)}
\right)^{1/m} =
\left(\lambda^2 \Lambda
\Gamma\left(\beta\right) t^\beta\right)^{1/m}
\left(
\frac{\Gamma\left((n-1)/a^*\right)}
{\Gamma(n/a^*)}
\right)^{1/m} \;.
\]
By the asymptotic expansion of the Gamma function (\cite[5.11.2, in p.
140]{NIST2010}),
\[
\frac{\Gamma\left((n-1)/a^*\right)}{\Gamma\left(n/a^* \right) }
\approx \left(\frac{e}{\beta}\right)^\beta
\left(1-\frac{1}{n}\right)^{(n-1)\beta} \frac{1}{n^\beta}
\approx \frac{1}{(\beta n)^\beta}\;,
\]
for large $n$. Clearly, $\beta >0$ since $1/a<1$. Hence for all $t>0$,
for
large $n$,
\[
\left(
\frac{B_n\left(t\;;\lambda\right)}
{B_{n-1}\left(t\;;\lambda\right)}
\right)^{1/m}
\approx
\left(\lambda^2\Lambda
\Gamma\left(\beta\right) t^\beta\right)^{1/m} \frac{1}{(\beta n)^{\beta/m}},
\]
and this goes to zero as $n\rightarrow+\infty$.
This completes the proof of Proposition \ref{P:K}.
\end{proof}

\begin{proof}[Proof of Proposition \ref{P:UpperBdd-K}]
The bound \eqref{E:UpBd-K-Mit} follows from the fact that
\begin{align}\label{E:Efrom1}
\sum_{k=1}^{\infty} \frac{z^k}{\Gamma(\alpha k)}
= z E_{\alpha,\alpha}(z)\;,
\end{align}
which can be easily seen from the definition, and the bound in Proposition \ref{P:K} (i):
\begin{align*}
\calK\left(t,x;\lambda\right)
&\le  \lMr{\delta}{G}{a}(t,x) \sum_{n=1}^\infty
B_{n}\left(t;\lambda\right)
=
\frac{1}{t}\lMr{\delta}{G}{a}(t,x) \sum_{n=1}^\infty
\frac{\left(\lambda^2\Lambda
\Gamma(1/a^*)\;t^{1/a^*}\right)^n}{\Gamma(n/a^*)}\\
&=
\lambda^2\Lambda \Gamma(1/a^*) t^{-1/a}
\lMr{\delta}{G}{a}(t,x)E_{1/a^*,1/a^*}\left(\lambda^2 \Lambda
\Gamma(1/a^*)
t^{1/a^*}\right).
\end{align*}
As for \eqref{E:UpBd-K}, we only need to show that the constant $C$ defined
in \eqref{E:Cst-UpK} is finite.
Let
\[
f(t) = \frac{E_{1/a^*,1/a^*}\left(\gamma\:
t^{1/a^*}\right)}{1+t^{1/a}\exp\left(\gamma^{a^*}\;
t\right)}\;.
\]
By Lemma \ref{L:Eab} with the real non-negative
value $z= \gamma\; t^{1/a^*}$ and $p=1$,
\[
\gamma E_{1/a^*,1/a^*}\left(\gamma\; t^{1/a^*}\right)\le
a^*\; \gamma^{a^*}\; t^{1/a}
\exp\left(\gamma^{a^*}\; t\right) +
O\left(\frac{1}{|t|^{2/a^*}}\right)
\:,\qquad t\rightarrow +\infty\:,
\]
where we have used the fact that $1/\Gamma(0)=0$,
we see that
\[
\lim_{t\rightarrow+\infty} f(t) \le a^*
\gamma^{a^*}\;.
\]
Since $E_{\alpha,\alpha}(\cdot)$ is continuous (by uniform convergence of the series in \eqref{E:Mittag-Leffler}),
we conclude that $\sup_{t\ge 0} f(t)<+\infty$.
This completes the proof of Proposition \ref{P:UpperBdd-K}.
\end{proof}

The next proposition is in principle a consequence of certain calculations in \cite{DebbiDozzi05On}. It is however not stated explicitly there,
so we include a proof for the convenience of the reader.

\begin{proposition}\label{P:G}
Fix $1<a\le 2$, $|\delta|\le 2-a$ and $1/a+1/a^*=1$.
There are three universal constants
\begin{align*}
C_1:=   \int_\R\frac{1-\cos(u)}{2 \pi \cos(\pi\delta/2) \vert u\vert^a}\,  \ud u\;,\quad
C_3:=\frac{a^* \Gamma(1+1/a)}{\pi\cos(2^{1/a}\pi\delta/2)^{1/a}}\;,
\quad C_2:= \left(2^{1/a^*}-1\right) C_3\;,
\end{align*}
such that
\begin{itemize}
 \item[(i)] for all $t>0$ and $x,y\in\R$,
\begin{align}\label{E:G-x}
 \int_0^t\ud r\int_\R \ud z
\left[\lMr{\delta}{G}{a}(t-r,x-z)-\lMr{\delta}{G}{a}(t-r,y-z)\right]^2
\le C_1 |x-y|^{a-1}\;;
\end{align}
 \item[(ii)] for all $s,t\in\R_+^*$ with $s\le t$, and $x\in\R$,
\begin{align}\label{E:G-t1}
 \int_0^s\ud r\int_\R \ud z
\left[\lMr{\delta}{G}{a}(t-r,x-z)-\lMr{\delta}{G}{a}(s-r,x-z)\right]^2
& \le C_2 (t-s)^{1-1/a}
\end{align}
and
\begin{align}\label{E:G-t2}
\int_s^t\ud r\int_\R \ud z \left[\lMr{\delta}{G}{a}(t-r,x-z)\right]^2
& \le C_3 (t-s)^{1-1/a}\;.
\end{align}
\end{itemize}
\end{proposition}

\begin{remark}
This proposition is a generalization of
\cite[Proposition \myRef{3.5}{P2:G}]{ChenDalang13Heat} for the heat equation.
In fact, if we take $a=2$ and
$\delta=0$, then $\lMr{\delta}{G}{a}(t,x) = G_2(t,x) = \frac{1}{\sqrt{4\pi
t}}\exp\left(-\frac{x^2}{4t}\right)$.
Let $C'_i$, $i=1,2,3$, be the optimal constants in \cite[Proposition
\myRef{3.5}{P2:G}]{ChenDalang13Heat} with
$\nu=2$. Then we have the following relation:
\[
C_1'=C_1 =\frac{1}{2}\;, \qquad C_2' = C_2=\frac{\sqrt{2}-1}{\sqrt{\pi}},\qquad C_3'=C_3=\frac{1}{\sqrt{\pi}}\;,
\]
where for $C_1$, we use the fact that $\int_\R\frac{1-\cos(u)}{u^2}\ud u=\int_\R\frac{\sin(u)}{u}\ud u=\pi$; see \cite[4.26.12, on p. 122]{NIST2010} for the last integral.
\end{remark}

\begin{proof}[Proof of Proposition \ref{P:G}]
(i) Note that
\[
\calF (\lMr{\delta}{G}{a}(t,\cdot)) (\xi) := \int_\R \ud x\: e^{-i\xi x}
\lMr{\delta}{G}{a}(t,x)
= \exp\left\{t\: \lMr{\delta}{\psi}{a}(\xi) \right\}
=
\exp\left\{- t |\xi|^a e^{- i \delta \pi \sgn(\xi)/2} \right\}.
\]
 By Plancherel's theorem, 
the left hand side of
\eqref{E:G-x} equals
\begin{align*}
\frac{1}{2\pi}\int_0^t\ud r\int_\R\ud \xi &\left|
e^{-i \xi x-(t-r)|\xi|^a e^{- i \delta \pi \sgn(\xi)/2}}
-
e^{-i \xi y-(t-r)|\xi|^a e^{- i \delta \pi \sgn(\xi)/2}}
\right|^2
 \\
& =
\frac{1}{2\pi}\int_0^t \ud r\int_\R \ud \xi
\: e^{-2(t-r)|\xi|^a\cos(\pi\delta/2)}
\left|e^{-i\xi x}-e^{-i\xi y}\right|^2\\
&=\frac{1}{\pi} \int_0^t \ud r\int_\R \ud \xi \:
e^{-2(t-r)|\xi|^a\cos(\pi\delta/2)}
\left(1-\cos(\xi(x-y))\right).
\end{align*}
After integrating over $r$,
the above integral equals
\[
\frac{1}{\pi}\int_\R \ud \xi\: \frac{1-e^{-2
t|\xi|^a\cos(\pi\delta/2)}}{2\cos(\pi\delta/2)|\xi|^a}
\left(1-\cos(\xi(x-y))\right).
\]
Use the change of variables $\xi = u/(x-y)$ to see that this is equal to
\[
  \frac{1}{\pi} |x-y|^{a-1} \int_\R \ud u\, \frac{1-\exp(-2t\vert u\vert^a \cos(\pi\delta/2) /\vert x-y\vert^a)}{2 \cos(\pi\delta/2) \vert u\vert^a} (1-\cos(u))
	  \leq C_1'\, |x-y|^{a-1},
\]
where
\[
   C_1' = \int_\R\frac{1-\cos(u)}{2 \pi \cos(\pi\delta/2) \vert u\vert^a}\,  \ud u .
\]
This proves \eqref{E:G-x}.



(ii) Denote the left hand side of \eqref{E:G-t1} by $I$. Apply Plancherel's
theorem for $I$:
\begin{align*}
 I=&\frac{1}{2\pi}\int_0^s\ud r\int_\R \ud \xi \left|
e^{-i\xi x-(t-r)|\xi|^a e^{-i\delta \pi \: \sgn(\xi)/2}}
-
e^{-i\xi x-(s-r)|\xi|^a e^{-i\delta \pi \: \sgn(\xi)/2}}
\right|^2 \\
=&
\frac{1}{2\pi}
\int_0^s\ud r\int_\R \ud \xi \left|
e^{-(t-r)|\xi|^a e^{-i\delta \pi \: \sgn(\xi)/2}}
-
e^{-(s-r)|\xi|^a e^{-i\delta \pi \: \sgn(\xi)/2}}
\right|^2
\end{align*}
Denote $\beta:=\pi\delta \sgn(\xi)/2$ and
\begin{align*}
 A_{r,t}&:= (t-r) |\xi|^a \cos(\beta),\qquad B_{r,t}:= (t-r) |\xi|^a \sin(\beta)\;.
\end{align*}
Then
\begin{align*}
 &\left|
e^{-(t-r)|\xi|^a e^{-i\delta \pi \: \sgn(\xi)/2}}
-
e^{-(s-r)|\xi|^a e^{-i\delta \pi \: \sgn(\xi)/2}}
\right|^2 \\
&\qquad= \left|e^{-A_{r,t}}\cos(B_{r,t}) + i e^{-A_{r,t}}\sin(B_{r,t}) -e^{-A_{r,s}}\cos(B_{r,s}) -
i e^{-A_{r,s}}\sin(B_{r,s}) \right|^2\\
&\qquad=e^{-2A_{r,t}} +e^{-2A_{r,s}} - 2 e^{-(A_{r,t}+A_{r,s})}\cos\left(B_{r,t}-B_{r,s}\right).
\end{align*}
Now, by the definition of $\Gamma(\cdot)$ function, we have that for all
$z\in \mathbb{C}$ with $\Re(z)>0$,
\begin{align}\label{E5_:GammaInt}
\int_\R \ud x\: e^{-z |x|^a} = \frac{2}{a} z^{-1/a} \int_0^\infty\ud y\: e^{-y}
y^{1/a-1}
=2 z^{-1/a} \Gamma\left(1+1/a\right).
\end{align}
Hence,
\begin{align}\label{E5_:At}
\int_\R \ud \xi\: e^{-2A_{r,t}}  = \int_\R \ud \xi\: e^{-2(t-r)\cos(\beta) |\xi|^a}
=\frac{2^{1/a^*}\Gamma\left(1+1/a\right)}{\cos(\beta)^{1/a}}
\frac{1}{(t-r)^{1/a}}\;.
\end{align}
Note that in the above integral, we have used the fact that
the value of $\cos(\beta)$ does not depend on $\xi$ because
$\cos\left(\beta\right) = \cos(\pi\delta/2)$.
Similarly,
\[
\int_\R\ud \xi\: e^{-2A_{r,s}}
=\frac{2^{1/a^*}\Gamma\left(1+1/a\right)}{\cos(\beta)^{1/a}}
\frac{1}{(s-r)^{1/a}}\;.
\]
For the third term, notice that
\begin{align*}
e^{-(A_{r,t}+A_{r,s})} \cos(B_{r,t}-B_{r,s}) &=
\exp\left(-\left(\frac{t+s}{2}-r\right)2\cos(\beta) |\xi|^a\right)\cdot
\cos\left((t-s)\sin(\beta)|\xi|^a\right) \\
&=
\Re\left[
\exp\left\{
-\left[\left(\frac{t+s}{2}-r\right)2\cos(\beta)
+i
(t-s)\sin(\beta) \right]|\xi|^a
\right\}
\right]
\end{align*}
Apply \eqref{E5_:GammaInt} with
$z=\left(\frac{t+s}{2}-r\right)\cos(\beta)
+i (t-s)\sin(\beta)$:
\begin{multline*}
\int_\R\ud \xi\:
\exp\left\{
-\left[\left(\frac{t+s}{2}-r\right)\cos(\beta)
+i
(t-s)\sin(\beta) \right]|\xi|^a
\right\} \\
=
2\Gamma(1+1/a) \left[
\left(\frac{t+s}{2}-r\right)2\cos(\beta)
+i
(t-s)\sin(\beta)
\right]^{-1/a}.
\end{multline*}
For $z\in \bbC$, suppose that $ z= \rho e^{i\theta}$ with $\theta \in\R$
and $\rho\ge 0$. For $c\le 0$, one has that
$|\Re(z^c)| = |\Re(\rho^c e^{i\theta c})|=\rho^c |\cos(\theta c)|\le \rho^c \le
|\Re(z)|^c$. Hence,
\[
2\int_\R  \ud \xi \: e^{-(A_{r,t}+A_{r,s})} \cos(B_{r,t}-B_{r,s}) \le
2^{1+1/a^*}\frac{\Gamma(1+1/a)}{\cos(\beta)^{1/a}}
\frac{1}{\left((t+s)/2-r\right)^{1/a}}\;.
\]
Integrating over $r$ and then applying Lemma \ref{L:TimeIncr} below, we
get (see the integration in \eqref{E5_:G-Dt2})
\begin{align}\notag
I &\le
\frac{\Gamma(1+1/a)}{2^{1/a} \pi\cos(\beta)^{1/a}}
\int_0^s\ud r
\left(\frac{1}{(t-r)^{1/a}}+\frac{1}{(s-r)^{1/a}}-\frac{2}{
\left[(t+s)/2-r\right]^ {1/a}}\right)  =
C_2 (t-s)^{1/a^*},
\end{align}
where $1/a^*+1/a=1$. As for \eqref{E:G-t2}, from \eqref{E5_:At}, we have
\begin{multline}
\int_s^t\ud r\int_\R \ud z \left[\lMr{\delta}{G}{a}(t-r,x-z)\right]^2
=
\frac{1}{2\pi}\int_s^t \ud r \int_\R \ud \xi\: e^{-(t-r) |\xi|^a \cos(\beta)}
\\ \label{E5_:G-Dt2}
= \frac{\Gamma(1+1/a)}{2^{1/a}\pi\cos(\beta)^{1/a}} \int_s^t\ud r\:
\frac{1}{(t-r)^{1/a}}
=
\frac{a^*\Gamma(1+1/a)}{2^{1/a}\pi\cos(\beta)^{1/a}} (t-s)^{1/a^*}\;.
\end{multline}
This completes the proof of Proposition \ref{P:G}.
\end{proof}

\begin{lemma}\label{L:TimeIncr}
For all $t\ge s\ge 0$, $a\in \;]1,2]$, we have
\[
\int_0^s\ud r\left(
\frac{1}{(t-r)^{1/a}}
+
\frac{1}{(s-r)^{1/a}}
-\frac{2}{((t+s)/2-r)^{1/a}}
\right)
\le a^* (2^{1/a}-1)\: \left(t-s\right)^{1/a^*}\;,
\]
where $a^*$ is the dual of $a$: $ 1/a+1/a^*=1$.
\end{lemma}
\begin{proof}
Clearly,
\begin{multline*}
\frac{1}{a^*}\int_0^s\ud r\left(
\frac{1}{(t-r)^{1/a}}
+
\frac{1}{(s-r)^{1/a}}
-\frac{2}{((t+s)/2-r)^{1/a}}
\right)  \\
=
s^{1/a^*} + t^{1/a^*} - (t-s)^{1/a^*}
+2^{1/a}(t-s)^{1/a^*}-2^{1/a}(t+s)^{1/a^*}\:.
\end{multline*}
We need to prove that
\[
\frac{s^{1/a^*} + t^{1/a^*} - (t-s)^{1/a^*}
+2^{1/a}(t-s)^{1/a^*}-2^{1/a}(t+s)^{1/a^*}}{(t-s)^{1/a^*}}
\]
is bounded from above for all $0\le s\le t$.
Or equivalently, we need to show that
\[
g(r):= \frac{r^{1/a^*} + 1 - (1-r)^{1/a^*}
+2^{1/a}(1-r)^{1/a^*}-2^{1/a}(1+r)^{1/a^*}}{(1-r)^{1/a^*}}
\]
is bounded for all $r\in [0,1]$.
Clearly, $g(0) = 0$ and $\lim_{r\uparrow 1} g(r) = 2^{1/a}-1$ (by applying
L'H\^opital's rule once).
Hence $\sup_{r\in [0,1]}g(r)<\infty$. Actually
\[
g'(r) =
\frac{\left((1+r)^{1/a}+(1+1/r)^{1/a}\right)-2^{1+1/a}}{a^*
(1-r)^{2-1/a}(1+r)^{1/a} }
\]
and notice that for all $r\in \;]0,1]$,
\[
(1+r)^{1/a}+(1+1/r)^{1/a} \ge 2\left[(1+r)(1+1/r) \right]^{1/(2a)}
= 2 \left(\sqrt{r}+\frac{1}{\sqrt{r}}\right)^{1/a} \ge 2^{1+1/a}\;.
\]
Hence  $g'(r)\ge 0$ for $r\in[0,1[$ and $\sup_{r\in [0,1]}g(r) = g(1)
=2^{1/a}-1$.
Therefore, Lemma \ref{L:TimeIncr} is proved with $C= a^* (2^{1/a}-1)$.
\end{proof}

The following proposition is useful to prove the $L^p(\Omega)$--continuity of $I(t,x)$.

\begin{proposition}\label{P:G-Margin}
Suppose that $a\in\:]1,2]$ and $|\delta|\le 2-a$.
Fix $(t,x)\in\R_+^*\times\R$. Denote
\begin{align*}
B:=B_{t,x}=\left\{(t',x')\in\R_+^*\times\R:\: 0 \le
t'\le
t+1/2,\: |x-x'|\le 1\right\}\:.
\end{align*}
Then there exists a constant $A>0$ such that for all
$(t',x')\in B$, $s \in [0,t'[$ and $|y|\ge A$,
\[
\lMr{\delta}{G}{a}\left(t'-s,x'-y\right) \le
\lMr{-\delta}{G}{a}\left(t+1-s,x-y\right)+\lMr{\delta}{G}{a}\left(t+1-s,x-y\right).
\]
\end{proposition}

\begin{proof}
The case where $a=2$ is proved in \cite[Proposition 5.3]{ChenDalang13Heat}, so
we only need to prove the case where $1<a<2$.
Denote $F(t,x):=\lMr{\delta}{G}{a}\left(t,x\right)+\lMr{-\delta}{G}{a}\left(t,x\right)$.
Suppose the mode of the density $\lMr{\delta}{G}{a}\left(1,x\right)$ is located at $m\in\R$.
By the scaling property, the mode of the density $\lMr{\delta}{G}{a}\left(t,x\right)$ locate at $t^{1/a}m$.
Hence, when $x\ge t^{1/a}|m|$ (resp. $x\le- t^{1/a}|m|$), the function $x\mapsto F(t,x)$ is decreasing (resp. increasing).

Fix $(t,x) \in \R_+^*\times\R$. Assume that $|y-x|>1+(t+1/2)^{1/a}|m|$. Because of the above fact, we know that
for all $(t',x')\in B$,
\begin{align}\label{E_:Gt'x'}
\lMr{\delta}{G}{a}\left(t'-s, x'-y\right)
\le
F\left(t'-s, |y-x|-|x-x'|\right)
\le
F\left(t'-s, |y-x|-1\right).
\end{align}
Apply Lemma \ref{L:Green} (v) with $N=1$ and use the scaling property of
$\lMr{\delta}{G}{a}(t,x)$ to get
\[
F(t,x) =2\frac{\Gamma(a+1)}{\pi} \sin(\pi a/2)\cos(\pi \delta/2) \frac{t}{|x|^{1+a}}
+ O\left(\frac{t^2}{|x|^{2a+1}}\right).
\]
Because $|\delta|\le 2-a$ and $a\in \:]1,2[\:$, we see that $ \sin(\pi a/2)\cos(\pi \delta/2) \ne 0$. Hence,
\begin{align*}
\frac{F\left(t+1-s,x-y\right)}{F\left(t'-s,
|y-x|-1\right)}
&=\frac{\frac{t+1-s}{|x-y|^{1+a}}+O\left(\frac{(t+1-s)^2}{|x-y|^{2a+1}}\right)}{\frac{t'-s}{||y-x|-1|^{1+a}}+O\left(\frac{(t'-s)^2}{||y-x|-1|^{2a+1}}\right)}\\
&=\frac{t+1-s}{t'-s} \frac{|x-y|^{-(a+1)}+O\left(\frac{t+1-s}{|x-y|^{2a+1}}\right)}{||y-x|-1|^{-(a+1)}+O\left(\frac{t'-s}{||y-x|-1|^{2a+1}}\right)}.
\end{align*}
Now it is clear that
\[
\lim_{|y|\rightarrow +\infty} \inf_{(t',x')\in B,\:s\in[0,t'[}
\frac{|x-y|^{-(a+1)}+O\left(\frac{t+1-s}{|x-y|^{2a+1}}\right)}{||y-x|-1|^{-(a+1)}+O\left(\frac{t'-s}{||y-x|-1|^{2a+1}}\right)}=1,
\]
which with \eqref{E_:Gt'x'} implies that
\begin{align*}
\lim_{|y|\rightarrow +\infty} \inf_{(t',x')\in B,\:s\in[0,t'[}\frac{F\left(t+1-s,x-y\right)}{F\left(t'-s,
x'-y\right)}
&\ge \inf_{(t',x')\in B,\:s\in[0,t'[}\frac{t+1 - s}{t + 1/2 - s} \\
&\geq \frac{t+1}{t+1/2} = 1+ \frac{1}{2t +1} >1,
\end{align*}
where we have used the fact that $s\mapsto (t+1 - s)/(t + 1/2 - s)$ is increasing.
Hence, we can choose a large constant $A$ uniformly over $(t',x')\in B$ and $s\in [0,t']$, such that for all $|y|\ge A$,
the inequality
\[
\frac{F\left(t+1-s,x-y\right)}{\lMr{\delta}{G}{a}\left(t'-s,
x'-y\right)}
\ge
1+ \frac{1}{2(t +1)}>1
\]
holds for all $(t',x')\in B$ and $s\in [0,t']$.
This completes the proof of Proposition \ref{P:G-Margin}.
\end{proof}


\begin{lemma}\label{L:DeriFrac}
 For all $m,n\in\bbN$, there exist polynomials
$\left\{P^{(n,m)}_{i}(x):\:i= 0,\dots,n\right\}$ such that
\begin{enumerate}[(1)]
 \item $P^{(n,m)}_{i}(x)$ are of degree $\leq i$ and they satisfy
\[
\frac{\partial^{n+m}}{\partial t^n \partial x^m} \lMr{\delta}{G}{a}(t,x)
= \frac{1}{(at)^{n}} \sum_{i=0}^{n} P_{i}^{(n,m)}(x)
\frac{\partial^{i+m}}{\partial x^{i+m}} \lMr{\delta}{G}{a}(t,x)\;;
\]
\item For fixed $t>0$, the partial derivative $\frac{\partial^{n+m}}{\partial
t^n \partial x^m} \lMr{\delta}{G}{a}(t,\cdot)$ as a function of $x$ is smooth
and integrable.
\end{enumerate}
\end{lemma}
\begin{proof}
Part (2) is a direct consequence of (1) and the upper bounds in Lemma
\ref{L:Green} (vi).
We now  prove (1).
It is clearly true for $n=m=0$: in this case,
$P_0^{(0,0)}(x)\equiv 1$. Moreover if $n=0$, then it is trivially true, with
$P_0^{(0,m)}(x) = 1$. 
Consider the case $n=1$ and $m=0$. Using the scaling properties twice, we have
\begin{align*}\notag
 \frac{\partial }{\partial t} \lMr{\delta}{G}{a}(t,x) =&
\left.\left[-\frac{1/a}{t^{1+1/a}} \lMr{\delta}{G}{a}\left(1,\xi\right) +
\frac{1}{t^{1/a}} \frac{\partial \lMr{\delta}{G}{a}(1,\xi)}{\partial \xi}
\frac{-x/a }{t^{1+1/a}} \right]\right|_{\xi=t^{-1/a}x}\\ \notag
=& -\frac{1}{a t}\left(\frac{1}{t^{1/a}}
\lMr{\delta}{G}{a}\left(1,\frac{x}{t^{1/a}}\right) + \left.\frac{x}{t^{2/a}}
 \frac{\partial \lMr{\delta}{G}{a}(1,\xi)}{\partial \xi}
\right|_{\xi=t^{-1/a}x}\right)\\
=& -\frac{1}{at}\left(\lMr{\delta}{G}{a}(t,x)+ x \frac{\partial
\lMr{\delta}{G}{a}(t,x)}{\partial x}\right).
\end{align*}
So in this case, $P^{(1,0)}_0(x)= -1$ and $P^{(1,0)}_1(x)= -x$.
Now suppose that it is true for $n, m\in\bbN$. It is easy to see that it is
true also for $n,m+1$ with
\[
P^{(n,m+1)}_i (x) =
P^{(n,m)}_i (x)+\frac{\ud}{\ud x}P^{(n,m)}_{i+1} (x) ,\quad \text{for $i=0,\dots,
n-1$},\qquad P^{(n,m+1)}_n (x) = P^{(n,m)}_n (x),
\]
so $P^{(n,m+1)}_i (x) $ is a polynomial of degree $\le i$.

   Now assume that $n\ge 1$ and the property is true for $\tilde n \leq n$ and all $m \geq 0$.
We shall establish the property for $n+1$ and $m$.
 By the induction assumption, we have
\[
\frac{\partial^{n+1+m}}{\partial t^{n+1} \partial x^m} \lMr{\delta}{G}{a}(t,x)
= \frac{-n a}{(at)^{n+1}} \sum_{i=0}^{n} P_{i}^{(n,m)}(x)
\frac{\partial^{i+m}}{\partial x^{i+m}} \lMr{\delta}{G}{a}(t,x)
+ \frac{1}{(at)^{n}} \sum_{i=0}^{n} P_{i}^{(n,m)}(x)
\frac{\partial^{1+i+m}}{\partial t \partial x^{i+m}} \lMr{\delta}{G}{a}(t,x).
\]
Then replace $\frac{\partial^{1+i+m}}{\partial t \partial x^{i+m}}
\lMr{\delta}{G}{a}(t,x)$ by the following sum using the induction assumption
\[
\frac{\partial^{1+i+m}}{\partial t \partial x^{i+m}} \lMr{\delta}{G}{a}(t,x)
= \frac{1}{at} \left(
P_0^{(1,i+m)}(x)  \frac{\partial^{i+m}}{\partial x^{i+m}}
\lMr{\delta}{G}{a}(t,x)
+
P_1^{(1,i+m)}(x) \frac{\partial^{i+m+1}}{\partial x^{i+m+1}}
 \lMr{\delta}{G}{a}(t,x)
\right).
\]
Finally, after grouping terms one can choose the following polynomials:
\begin{align*}
 P_0^{(n+1,m)}(x)
=& -n a P_0^{(n,m)}(x) + P_0^{(n,m)}(x)P_0^{(1,m)}(x),
\end{align*}
which is a polynomial of order $0$,
\begin{align*}
P_i^{(n+1,m)}(x)  =&
-na P_i^{(n,m)}(x) + P_i^{(n,m)}(x)P_0^{(1,i+m)}(x)+
P_{i-1}^{(n,m)}(x)P_1^{(1,i+m-1)}(x),
\end{align*}
which are polynomials of degree $\leq i$, for $i=1,\dots, n$, and
\begin{align*}
P_{n+1}^{(n+1,m)}(x) =&  P_n^{(n,m)}(x) P_1^{(1,n+m)}(x),
\end{align*}
which are polynomials of degree $\le n+1$. This completes the proof of Lemma
\ref{L:DeriFrac}.
\end{proof}
%


\begin{lemma}\label{L:InDt}
Suppose that $a\in \;]1,2]$ and  $\mu\in \calM_a\left(\R\right)$.
\begin{enumerate}[(1)]
 \item The function $J_0(t,x) = \left(\lMr{\delta}{G}{a}(t,\cdot)
*\mu\right)(x)$
belongs to $C^\infty\left(\R_+^*\times\R\right)$.
\item
For all compact sets $K\subset \R_+^*\times\R$ and $v\in\R$,
\begin{align}
\label{E:InDt}
\sup_{(t,x)\in K}\left(\left[v^2+J_0^2\right]\star \calK \right)(t,x)
<\infty.
\end{align}
In particular,
\begin{align}\label{E:J20K}
\left(J_0^2\star \calK \right)(t,x)  \le
C' (t\vee 1)^{2(1+1/a)}  t^{1-2/a} \left[
t^{-1/a} +  \exp\left(\gamma^{a^*} t\right)
\right],
\end{align}
where
\begin{align}\label{E:InDt-Cst}
 C':= C A_a^2 K_{a,0}^2
\max\left(a^*,\frac{\Gamma(1/a^*)^2}{\Gamma(2/a^*)}\right),
\end{align}
$C=C(a,\delta,\lambda)$ is defined in \eqref{E:Cst-UpK}, $K_{a,0}$ is defined in \eqref{E:G-bd}, and
\begin{align}
 \label{E:Aa}
A_a:=\sup_{y\in\R}\int_\R \frac{|\mu|(\ud z)}{1+|y-z|^{1+a}}.
\end{align}
\end{enumerate}
\end{lemma}


\begin{proof}
(1) Fix $0<t\le T$ and $n,m \in\bbN$. By Lemma \ref{L:DeriFrac} and \eqref{E:Gtx-bd},
\[
\left|
\frac{\partial^{n+m}}{\partial t^n \partial x^m} \lMr{\delta}{G}{a}(t,x)\right|
\le
 \frac{1}{(at)^{n}} \sum_{i=0}^{n}  \left|P_{i}^{(n,m)}(x)\right| K_{a,i+m}\: t^{-\frac{i+m+1}{a}}
\frac{(T\vee 1)^{1+\frac{i+m+1}{a}}}{1+|x|^{1+i+m+a}}.
\]
Since the polynomials $P_{i}^{(n,m)}(x)$ are of order $i$, for some finite constant $C>0$ depending on $a$, $m$, $n$ and $T$, the above bound reduces to
\[
\left|
\frac{\partial^{n+m}}{\partial t^n \partial x^m} \lMr{\delta}{G}{a}(t,x)
\right|
\le
 C \frac{ g(t)}{1+ |x|^{m+1+a}},\quad\text{with $g(t):=\sum_{i=0}^n t^{-n-\frac{1+i+m}{a}}$}.
\]
Hence, for $0<t_1<t_2\le T$,
\begin{equation}\label{rd1}
\int_{t_1}^{t_2}\ud s \int_\R \mu(\ud z) \left|\frac{\partial^{n+m}}{\partial t^n \partial x^m} \lMr{\delta}{G}{a}(s,z)\right|<+\infty.
\end{equation}
 By Fubini's theorem and induction, it is now possible to conclude that $J_0(\cdot,\circ)\in C^{\infty}(\R_+^*\times\R)$. Indeed, the first step of this induction argument is:
\begin{align*}
   J_0(t_2,x) - J_1(t_1,x) &= \int_\R \mu(\ud y) (\lMr{\delta}{G}{a}(t_2,x-y) - \lMr{\delta}{G}{a}(t_1,x-y))\\
	  &= \int_\R \mu(\ud y) \int_{t_1}^{t_2} \ud t \, \frac{\partial}{\partial t} \lMr{\delta}{G}{a}(t,x-y)
	   = \int_{t_1}^{t_2} \ud t  \int_\R \mu(\ud y) \frac{\partial}{\partial t} \lMr{\delta}{G}{a}(t,x-y),
\end{align*}
where we have used Fubini's theorem, which applies by \eqref{rd1}. This shows that
$$
    \frac{\partial}{\partial t} J_0(t,x) =  \int_\R \mu(\ud y) \frac{\partial}{\partial t} \lMr{\delta}{G}{a}(t,x-y),
$$
and higher derivatives are obtained by induction. This proves (1).

(2) Without loss of generality,
assume that $\mu$ is non-negative, i.e., $\mu\in \calM_{a,+}\left(\R\right)$.
By \eqref{E:Gtx-bd}, for $0<s\le t$,
\begin{align}\label{E:J0B}
J_0\left(s,y\right) \le  A_a\: K_{a,0}\: (t\vee 1)^{1+1/a} s^{-1/a},
\end{align}
where $A_a$ is defined in \eqref{E:Aa}.
Hence, by \eqref{E:UpBd-K}, and by replacing one factor $J_0(s,y)$ of $J_0^2(s,y)$ by the above bound, we have that
\begin{multline*}
\left(J_0^2\star \calK\right)(t,x) \le C \int_0^t \ud
s\left(\frac{1}{(t-s)^{1/a}}+ \exp\left(\gamma^{a^*}(t-s)\right)
\right)\int_\R \ud y\:
\lMr{\delta}{G}{a}\left(t-s,x-y\right) \\
\times A_a K_{a,0} (t\vee 1)^{1+1/a} s^{-1/a} \int_\R \mu(\ud z) \lMr{\delta}{G}{a}(s,y-z),
\end{multline*}
where the constant $C:=C(a,\delta,\lambda)$ is defined in
\eqref{E:Cst-UpK}.
Integrate over $\ud y$ using the semigroup property, and then integrate over $\mu(\ud z)$:
\begin{align}\label{E:Jsq0K}
\left(J_0^2\star \calK\right)(t,x) \le C A_a K_{a,0} (t\vee 1)^{1+1/a} J_0(t,x) \int_0^t \ud
s \: \frac{1}{s^{1/a}}\left(\frac{1}{(t-s)^{1/a}}+ \exp\left(\gamma^{a^*}(t-s)\right)
\right).
\end{align}
Apply \eqref{E:J0B} to $J_0(t,x)$.
The integral over $s$ gives
\begin{align}\notag
\int_0^t \ud s\left(\frac{1}{s^{1/a}(t-s)^{1/a}}+
\frac{1}{s^{1/a}}\exp\left(\gamma^{a^*}(t-s)\right)
\right)  & \le
\int_0^t \ud s\left(\frac{1}{s^{1/a}(t-s)^{1/a}}+
\frac{1}{s^{1/a}}\exp\left(\gamma^{a^*}t\right)
\right) \\ \notag
&= t^{1-2/a}\frac{\Gamma\left(1-1/a\right)^2}{\Gamma\left(2-2/a\right)}
+ a^* t^{1/a^*} \exp\left(\gamma^{a^*}t\right)\\
&=
t^{1/a^*} \left(
\frac{1}{t^{1/a}} \frac{\Gamma\left(1/a^*\right)^2}{\Gamma\left(2/a^*\right)}
+ a^* \exp\left(\gamma^{a^*}t\right)\right).
\label{E:IntS}
\end{align}
Hence, combining the above facts proves \eqref{E:J20K}.
This completes the proof of Lemma \ref{L:InDt}.
\end{proof}

\begin{proof}[Proof of Theorem \ref{T:ExUni}]
The proof follows the same six steps as those in the proof of
\cite[Theorem 2.4]{ChenDalang13Heat} with some minor changes:

(1) Both proofs rely on estimates on the kernel function $\calK(t,x)$ . Instead of an explicit formula as for the heat equation case (see
\cite[Proposition 2.2]{ChenDalang13Heat}), Proposition \ref{P:UpperBdd-K} ensures the finiteness and provides a bound on the kernel function $\calK(t,x)$.

(2) In the Picard iteration scheme (i.e., Steps 1--4 in the proof of \cite[Theorem 2.4]{ChenDalang13Heat}),
we need to check the $L^p(\Omega)$-continuity of the stochastic integral,
which then guarantees that at the next step, the integrand is again in $\calP_2$, via \cite[Proposition 3.4]{ChenDalang13Heat}.
Here, the statement of \cite[Proposition 3.4]{ChenDalang13Heat} is still true by replacing in its proof
\cite[Propositions 3.5 and 5.3]{ChenDalang13Heat} by Propositions \ref{P:G} and \ref{P:G-Margin}, respectively.
Note that when applying Proposition \ref{P:G-Margin}, we need to replace the $G_\nu^2$ in \cite[(3.8)]{ChenDalang13Heat}
by $(\lMr{-\delta}{G}{a}+\lMr{\delta}{G}{a})^2\le 2\lMr{-\delta}{G}{a}^2+2\lMr{\delta}{G}{a}^2$.

(3) In the first step of the Picard iteration scheme, the following property is useful:
For all compact sets $K\subseteq \R_+\times\R$,
\[
\sup_{(t,x)\in K}\left(\left[1+J_0^2\right]\star
\lMr{\delta}{G}{a}^2 \right) (t,x)<+\infty.
\]
For the heat equation, this property is discussed in \cite[Lemma 3.9]{ChenDalang13Heat}.
Here, Lemma \ref{L:InDt} gives the desired result with minimal requirements on the initial data.
This property, together with the calculation of  the upper bound on the function $\calK$ in Proposition \ref{P:UpperBdd-K}, guarantees that
all the $L^p(\Omega)$-moments of $u(t,x)$ are finite.
This property is also used to establish uniform convergence of the Picard iteration scheme, hence $L^p(\Omega)$--continuity of $(t,x)\mapsto I(t,x)$.

The proofs of \eqref{E:SecMom-Lower} and \eqref{E:TP-Lower} are identical to those of the corresponding properties in \cite[Theorem 2.4]{ChenDalang13Heat}, and \eqref{E:SecMom-Lin} and
\eqref{E:TP-Lin} are direct consequences of the preceding statements.

This completes the proof of Theorem \ref{T:ExUni}.
\end{proof}

\section{Proofs of Theorems \ref{T:Intermittency} and \ref{T:Growth}}
\label{S:IntGrow}

We begin with the upper bound in Theorem \ref{T:Intermittency}.

\begin{proof}[Proof of Theorem \ref{T:Intermittency} (1)]
  Recall from \eqref{E:D-gamma} that $\widehat{\gamma}_p= a_{p,\Vip}^2 z_p^2 \Lip_\rho^2 \Lambda\Gamma(1/a^*)$, and
$a^*=a/(a-1)$.
By \eqref{E:SecMom-Up}, \eqref{E:J20K} and \eqref{E:J0B}, for all $x\in\R$,
\[
\overline{m}_p(x) =
\mathop{\lim\sup}_{t\rightarrow\infty}\frac{\log\Norm{u(t,x)}_p^p}{t}
\le \frac{\widehat{\gamma}^{a^*}p }{2} = \frac{p}{2} \left(a_{p,\Vip}^2 z_p^2\Lip_\rho^2
\Lambda\Gamma(1-\frac{1}{a})\right)^{a/(a-1)} \;.
\]
Since $a_{p,\Vip}\le 2$ and $z_p\le 2\sqrt{p}$, \eqref{E:InterU} follows.
\end{proof}

\subsection{Lower bound on $\calK(t,x)$ (Proposition \ref{P:LowBdd-K})}
\label{SS:LowBd-K}

\begin{lemma} \label{L:Cad}
Suppose that  $a\in \;]1,2[$ and $|\delta|< 2-a$. Then the constant
$\widetilde{C}_{a,\delta}$ defined in \eqref{E:Cad} is strictly positive,
and so
\begin{align}\label{E:GlowBd}
\lMr{\delta}{G}{a}(t,x)\ge \widetilde{C}_{a,\delta} \; \pi \, g_a(t,x) =  \frac{\widetilde{C}_{a,\delta} \;t}{\left(t^{2/a}+x^2
\right)^{\frac{a}{2}+\frac{1}{2}}},\quad\text{for all $t>0$ and $x\in\R$.}
\end{align}
\end{lemma}
\begin{proof}
By the scaling property of both $\lMr{\delta}{G}{a}$
and $g_a(t,x)$,
\[
 \inf_{(t,x)\in\R_+^*\times\R}\;\;
\frac{\lMr{\delta}{G}{a}(t,x)}{\pi g_a(t,x)}
= \inf_{y\in\R}\;\; \frac{\lMr{\delta}{G}{a}(1,y)}{\pi g_a(1,y)}\;.
\]
Let $f(y)=\frac{\lMr{\delta}{G}{a}(1,y)}{\pi g_a(1,y)}$.
Because the support of $\lMr{\delta}{G}{a}(1,\circ)$ is $\R$ (see \cite[Remark 4, p.79]{Zolotarev86}),
$f(y)>0$ for all $y\in\R$.
In the case where $1<a\le 2$ and $|\delta|< 2-a$, both
$\lMr{\delta}{G}{a}(1,y)$ and $g_a(1,y)$ have tails at $\pm\infty$ with
polynomial decay of the same rate as $|y|^{-1-a}$: see \cite[p.143]{Zolotarev86} (we use here the fact that $|\delta|\ne 2-a$).
Hence,
\[
\lim_{y\rightarrow\pm \infty} f(y)>0\;.
\]
Therefore, $f(y)$ is a smooth function on $\overline{\R}:=\R\cup\{\pm\infty\}$
such that $f(y)>0$ for all $y\in\overline{\R}$. This implies that
$\inf_{y\in\R} f(y)>0$, which completes the proof of Lemma \ref{L:Cad}.
\end{proof}

\begin{lemma}\label{L:Fourier}
Let $f(x) = \left(b^2+x^2\right)^{-\nu-1/2}$ with $b>0$ and $\nu \ge 1/2$.
Then
\begin{align}\label{E:Ff-Low}
\calF[f] (z) = \int_\R\ud x\: e^{-i z x} f(x)   \ge
C_\nu \:
b^{-2\nu} \exp\left(-b |z| \right)
\:,
\end{align}
for all $b>0$ and $z\in\R$, where the constant $C_\nu>0$ is given in \eqref{E:Cnu}.
\end{lemma}

\begin{proof}
Note that the function $f(x)$ is an even function, so its Fourier
transform is a real-valued function, instead of a complex
one, which allows us to bound this transform from below. Indeed, by \cite[(7) p.11]{Erdelyi1954-I},
we have that
\[
 \calF[f] (z)=
 \left(
\frac{|z|}{b}\right)^\nu \frac{\sqrt{\pi}}{2^{\nu} \Gamma\left(\nu+1/2\right)}
K_\nu\left(b |z|\right),\quad\text{for $\Re(b)>0$ and $\nu>-1/2$,}
\]
where $K_\nu(x)$ is the modified Bessel function of the second kind.
Equivalently, we need to prove that the function
\[
\R_+\times\R  \ni (b,z) \mapsto \left(
\frac{|z|}{b}\right)^\nu \frac{\sqrt{\pi}}{2^{\nu} \Gamma\left(\nu+1/2\right)}
K_\nu\left(b |z|\right) b^{2\nu} \exp\left(b|z|\right)
\]
is uniformly bounded away from zero. By choosing $u=b|z|$, we reduce this
problem to bounding the following function
\begin{align}\label{E5_:Fourier}
\R_+\ni u \mapsto \frac{\sqrt{\pi}}{2^\nu
\Gamma\left(\nu+1/2\right)}
f(u)
\end{align}
away from zero, where $f(u) := u^\nu e^{u} K_\nu(u)$.
By the differential formula for $x^{\pm\nu}K_\nu(x)$ (see, e.g.,
\cite[51:10:4, p.532]{oldham2008atlas}),
\[
f'(u) = e^u u^\nu \left( K_\nu(u) - K_{\nu-1}(u) \right).
\]
By the integral representation of $K_\nu(z)$ in \cite[10.32.9, p.
252]{NIST2010},
\begin{align*}
K_\nu(u) - K_{\nu-1}(u)
&=
\frac{1}{2} \int_0^\infty e^{-u\cosh(t)} \left(e^{\nu
t}-e^{-(\nu-1)t}\right)\left(1-e^{-t}\right)\ud t\ge 0\:.
\end{align*}
Hence, $f'(u)>0$ and
\[
\inf_{u\in\R_+}f(u) = \lim_{u\rightarrow 0} f(u) = 2^{\nu-1}\Gamma(\nu)\;,
\]
where we have used the property $K_\nu(u) \sim \frac{1}{2} \Gamma(\nu) (\frac{1}{2} u)^{-\nu}$ as $u\downarrow 0$ (see \cite[10.30.2, p. 252]{NIST2010}).
Therefore,
\[
C_\nu = \inf_{u\in\R_+}
\frac{\sqrt{\pi}}{2^\nu
\Gamma\left(\nu+1/2\right)}
f(u)
 = \frac{\Gamma(\nu)\Gamma(1/2)}{2 \Gamma(\nu+1/2)}\;,
\]
This completes the proof of Lemma \ref{L:Fourier}.
%
\end{proof}

\begin{lemma}\label{L:stst}
 For all $x\in\R$, $0\le s\le t$ and $a\in \;]1,2]$, we have
\[
g_1\left(s^{1/a}+(t-s)^{1/a},x\right)\ge
\frac{\sqrt{2}}{2}
g_1\left(t^{1/a},x\right).
\]
\end{lemma}
\begin{proof}
First notice that
\[
s^{1/a}+(t-s)^{1/a} = t^{1/a}
\left((s/t)^{1/a}+(1-s/t)^{1/a}\right).
\]
Elementary calculations show that the
function $f(r) = r^{1/a}+(1-r)^{1/a}$ satisfies $1\le f(r)\le 2 (1/2)^{1/a}$ when $r\in [0,1]$.
Since $a\in \;]1,2]$, the upper bound
is bounded further by $\sqrt{2}$. Hence,
\begin{align}\label{E:tst}
t^{1/a} \le
s^{1/a}+(t-s)^{1/a}
\le  \sqrt{2}\, t^{1/a}\:.
\end{align}
We need a property of $g_1(t,x)$: If $0<t_0\le t\le t_1$, then
\begin{align}\label{E:stst}
g_1(t,x)\ge \min\left(g_1(t_0,x),g_1(t_1,x)\right).
\end{align}
Indeed, we only need to show \eqref{E:stst} for $x\ne 0$. When $x\ne 0$,
the function $t\mapsto g_1(t,x)$
is increasing on $t\in [0,x]$ and decreasing on $t\in[x,\infty[\:$ because
$\frac{\partial g_1}{\partial t}(t,x) \geq 0$ iff $0 \le t \le x$.
Hence, \eqref{E:stst} holds. Therefore, \eqref{E:tst} implies that
\begin{align} \label{E5_:stst}
\frac{g_1\left(s^{1/a}+(t-s)^{1/a},x\right)}{g_1\left(t^{1/a},x\right)}
\ge \min\left(1,
\frac{g_1\left(\sqrt{2}\:t^{1/a},x\right)}{g_1\left(t^{1/a},x\right)}\right).
\end{align}
Notice that
\[
\frac{g_1\left(\sqrt{2}\:t^{1/a},x\right)}{g_1\left(t^{1/a},x\right)}
= \frac{\sqrt{2}\left(t^{2/a}+x^2\right)}{2t^{2/a}+x^2}
= \sqrt{2} -\frac{\sqrt{2} t^{2/a}}{2t^{2/a} + x^2}
\ge \sqrt{2} - \frac{\sqrt{2}}{2} = \frac{\sqrt{2}}{2}\;.
\]
Taking this lower bound back to \eqref{E5_:stst} proves Lemma \ref{L:stst}.
\end{proof}



\begin{proof}[Proof of Proposition \ref{P:LowBdd-K}.]
Notice that, by \eqref{E:K} and \eqref{E:GlowBd},
\[
\calK(t,x) = \sum_{n=1}^\infty
\left(\lambda^2\lMr{\delta}{G}{a}^2\right)^{\star n}(t,x)
\ge \sum_{n=1}^\infty \left(\lambda^2 \widetilde{C}^2_{a,\delta}\; \pi^2\;
g_a^2\right)^{\star n} (t,x) \ge
\sum_{n=1}^\infty \left(\lambda^2 \widetilde{C}^2_{a,\delta}\; \pi^2\;
g_a^2\right)^{\star 2n}(t,x)\;.
\]
We now calculate space-time convolutions of $g_a^2$. By Plancherel's theorem and
Lemma \ref{L:Fourier}
with $\nu=a+1/2$ and $b=t^{1/a}$, we have that
\begin{multline*}
\int_0^t\ud s\int_\R \ud y \; g_a^2\left(t-s,x-y\right) g_a^2\left(s,y\right)
= \frac{1}{2\pi} \int_0^t\ud s\int_\R \ud z\;
\calF\left[g_a^2(t-s,x-\cdot)\right](z)\;
\calF\left[g_a^2(s,\cdot)\right](z) \\
\ge \frac{C_{a+1/2}^2}{2\pi^5}\int_0^t \ud s\; s^2(t-s)^2
\frac{1}{s^{2+1/a}}
\frac{1}{(t-s)^{2+1/a}}
\int_\R \ud z \exp\left(-i x z-|z|
\left(s^{1/a}+(t-s)^{1/a}\right)\right).
\end{multline*}
By the formula $\int_0^\infty \ud z\: \cos(xz) e^{-z} = (1+x^2)^{-1}$ (which explicits the Laplace transform of $\cos(xz)$) for
$x\in\R$ and the bound in Lemma
\ref{L:stst}, the $\ud z$-integral satisfies:
\begin{align*}
\int_\R \ud z\: \exp\left(-i x z-|z|
\left(s^{1/a}+(t-s)^{1/a}\right)\right)
&=
 2 \int_0^\infty \ud z\: \cos(xz) \exp\left(-z
\left(s^{1/a}+(t-s)^{1/a}\right)\right)\\
&=2 \pi g_1\left(s^{1/a}+(t-s)^{1/a}, x\right)\\
&\ge \pi \sqrt{2}
g_1\left(t^{1/a}, x\right).
\end{align*}
As for the integral over the time variable, using the Euler's Beta integral
\eqref{E:BI}, we have
\begin{align*}
\int_0^t\ud s\: \left[s(t-s)\right]^{ - 1/a}=
\frac{\Gamma\left(1-1/a\right)^2}{\Gamma\left(2-2/a\right)}
\: t^{1-2/a}\;.
\end{align*}
With these calculations, we obtain
\[
\left(g_a^2\star g_a^2\right)(t,x) \ge
K_1
 \:t^{1-2/a}\: g_1\left(t^{1/a},x\right),
\]
with
\[
K_1 :=
\frac{C_{a+1/2}^2\; \Gamma\left(1-1/a\right)^2 }{\pi^4
\sqrt{2}\;\; \Gamma\left(2-2/a\right)}\;.
\]
Denote
\[
\left(g_a^{2}\right)^{\star n} (t,x):= \underbrace{\left(g_a^2\star\cdots\star
g_a^2 \right)}_{\text{$n$ factors}}(t,x)\;.
\]
By the above calculation, we know that
\[
\left(g_a^{2}\right)^{\star 2} (t,x) \ge K_1 \:  t^{1-2/a}
g_1\left(t^{1/a},x\right).
\]
Suppose by induction that all $n\in\bbN$,
\[
\left(g_a^{2}\right)^{\star 2n} (t,x) \ge K_n \:  t^{2n-1-2n/a}
g_1\left(t^{1/a},x\right).
\]
Then
\begin{align*}
\left(g_a^{2}\right)^{\star 2(n+1)} (t,x)
&= \left(\left(g_a^{2}\right)^{\star 2 n}\star \left(g_a^2\right)^{\star 2}
\right)(t,x)\\
&\ge K_n\, K_1 \:
\int_0^t\ud s\: s^{2n-1-2n/a}(t-s)^{1-2/a}
\left(g_1\left(s^{1/a},\cdot\right)*g_1\left((t-s)^{1/a},
\cdot\right)\right)(x)\;.
\end{align*}
Using the semigroup property of $g_1$ and Lemma \ref{L:stst},
\[
\left(g_1\left(s^{1/a},\cdot\right)*g_1\left((t-s)^{1/a},\cdot\right)\right)(x)
=
g_1\left(s^{1/a}+(t-s)^{1/a},x\right)
\ge
\frac{1}{\sqrt{2}} \:g_1\left(t^{1/a},x\right).
\]
The $\ud s$-integral gives, by Euler's Beta integral \eqref{E:BI},
\[
\int_0^t \ud s\: s^{2n-1-2n/a}(t-s)^{1-2/a}
=\frac{\Gamma(b)\Gamma(b n)}{\Gamma(b(1+n))}\: t^{2(n+1)-1-2(n+1)/a}\;,
\]
where $b=2-2/a$. Thus we have
\[
\left(g_a^{2}\right)^{\star 2(n+1)} (t,x) \ge K_{n+1}
 \: t^{(n+1)b-1}\: g_1\left(t^{1/a},x\right),
\]
with the constant
\[
K_{n+1} = \frac{K_n}{2^{1/2}} \frac{K_1\Gamma(b)\Gamma(b
n)}{\Gamma(b(1+n))}
=\frac{K_{n-1}}{2^{2/2}}\frac{K_1\Gamma(b)\Gamma(b
(n-1))}{\Gamma(bn)}\frac{K_1\Gamma(b)\Gamma(b n)}{\Gamma(b(1+n))}=\cdots
=\frac{K_1^{n+1}}{2^{n/2}} \frac{\Gamma(b)^{n+1}}{\Gamma(b(1+n))}
\:.
\]
Therefore, we have
\begin{align*}
\calK(t,x) &\ge  \sqrt{2}\:  \sum_{n=1}^\infty
\frac{\left[\lambda^4 \widetilde{C}^4_{a,\delta} \pi^4 K_1\Gamma(b)\right]^n}{2^{n/2}\Gamma(b n)} \: t^{nb-1}\:
g_1\left(t^{1/a},x\right)\\
&= \sqrt{2}\; g_1\left(t^{1/a},x\right) t^{-1}
\sum_{n=1}^\infty \frac{\Upsilon^n  t^{bn}}{\Gamma(bn)}\\
&=
 \sqrt{2}\; \Upsilon\: g_1\left(t^{1/a},x\right) t^{b-1}
E_{b,b}\left(\Upsilon t^b\right),
\end{align*}
 where $\Upsilon:=\Upsilon(\lambda,a,\delta)= 2^{-1/2}\lambda^4\:\widetilde{C}^4_{a,\delta} \pi^4 K_1\Gamma(b)$
and
in the last equation we have used \eqref{E:Efrom1}.
The constant $C = C(\lambda, a, \delta)$ can be chosen as
\[
C = \sqrt{2}\: \Upsilon(\lambda,a,\delta)=
 2^{-1/2}
\:\lambda^4\:
\widetilde{C}^4_{a,\delta} \:C_{a+1/2}^2\: \Gamma(1-1/a)^2 ,
\]
which completes the proof of the lower bound on $\calK(t,x)$.

Using the fact that $\int_\R g_1\left(t^{1/a},x\right)\ud x = 1$, we have
\[
\int_\R\ud y\: \calK(t,y) \ge C\:  t^{b-1}
E_{b,b}\left(\Upsilon t^b\right).
\]
Recall that $b=2-2/a$ and so $b\in \;]0,1]$.
Integrating term-by-term in \eqref{E:Mittag-Leffler}, we
obtain
\begin{align}\label{E:IntE}
\int_0^t\ud s\: E_{\alpha,\beta}\left(\lambda s^\alpha \right) s^{\beta-1} =
t^\beta E_{\alpha,\beta+1}\left(\lambda t^\alpha\right),\quad \beta>0\:;
\end{align}
see \cite[(1.99) on p.24]{Podlubny99FDE}.
Therefore, integrating over $s$ using \eqref{E:IntE}, we see that
\begin{align*}
\int_0^t\ud s \int_\R \calK\left(s,y\right)\ud y &\ge
C\int_0^t \ud s\:  s^{b-1} E_{b,b}\left(\Upsilon s^b\right)=C \: t^b\: E_{b,b+1}\left(\Upsilon t^b \right),
\end{align*}
which completes the proof of Proposition \ref{P:LowBdd-K}.
\end{proof}

\subsection{Proofs of Theorems \ref{T:Growth} and \ref{T:Intermittency} (2)}
\label{SS:Growth}

We need some properties of $g_a(t,x)$ defined in \eqref{E:ga}.
\begin{lemma}\label{L:gaLowB}
For $a>0$, $g_a(t,x-y) \ge \pi t^{1/a} \:
g_a\left(t,\sqrt{2} \: x\right)
g_a\left(t,\sqrt{2} \:y\right)$.
\end{lemma}
\begin{proof}
This is a consequence of the inequalities $1+(x-y)^2\le 1+2x^2+2y^2\le (1+2x^2)(1+2y^2)$.
\end{proof}

\begin{lemma}\label{L:J0LowB}
Suppose that $a\in\;]1,2[$, $|\delta| < 2-a$ and $\mu\in
\calM_{a,+}\left(\R\right)$, $\mu\ne 0$. Then for all $\epsilon>0$, there exists
a constant $C$ such that
\[
\left(\lMr{\delta}{G}{a}(t,\cdot)*\mu\right)(x) \ge C \:
1_{[\epsilon,\infty[}(t)\, g_a\left(t,\sqrt{2}\: x\right),\qquad\text{for all
$t\ge 0$ and
$x\in\R$.}
\]
\end{lemma}
\begin{proof}
Denote $J_0(t,x)=\left(\lMr{\delta}{G}{a}(t,\cdot)*\mu\right)(x)$.
By the lower bound on $\lMr{\delta}{G}{a}(t,x)$  in \eqref{E:GlowBd},
Lemma \ref{L:gaLowB} and the scaling property of $g_a(t,x)$, we have
\begin{align*}
J_0(t,x) &\ge
\widetilde{C}_{a,\delta}\, \pi \int_\R\mu(\ud y)\: g_a(t,x-y)
\\
&\ge
\widetilde{C}_{a,\delta}\;\pi^2 t^{1/a}
g_a\left(t,\sqrt{2}\:x\right)
\int_\R\mu(\ud y)\:
g_a\left(t,\sqrt{2}\:y\right) \\
& =
\widetilde{C}_{a,\delta}\;\pi^2
g_a\left(t,\sqrt{2}\:x\right)
\int_\R \mu(\ud y)\:
\left(1+2\frac{y^2}{t^{2/a}}\right)^{
-\frac{a+1}{2}}\:.
\end{align*}
The above integrand is non-decreasing with respect to $t$. Hence
\begin{align*}
J_0(t,x) &\ge
\widetilde{C}_{a,\delta}\; \pi^2 \:
\Indt{t\ge \epsilon} \:
g_a\left(t,\sqrt{2}\:x\right)
\int_\R \mu(\ud y) \:
\left(1+2\frac{y^2}{\epsilon^{2/a}}\right)^{
-\frac{a+1}{2}}\\
&=
\widetilde{C}_{a,\delta}\; \pi^2 \:\epsilon^{1/a}
\Indt{t\ge \epsilon} \:
g_a\left(t,\sqrt{2}\:x\right)
\int_\R\mu(\ud y)\: g_a\left(\epsilon,\sqrt{2}\: y\right).
\end{align*}
Since the function $y\mapsto
g_a\left(\epsilon,\sqrt{2}y\right)$ is strictly positive
and $\mu$ is nonnegative and non-vanishing, the integral is positive.
Finally, we can take $C :=
\widetilde{C}_{a,\delta}\;\pi^2 \:\epsilon^{1/a}
\int_\R\mu(\ud y)\: g_a\left(\epsilon,\sqrt{2}\: y\right)$.
\end{proof}

\begin{lemma}\label{L:ga2g1}
For all $a>0$, $t\ge s\ge 0$ and $x\in\R$, we have
\[
\left(g_a^2\left(t-s,\sqrt{2} \;\cdot\right) *
g_1\left(s^{1/a},\cdot\right)\right)(x)
\ge
\frac{\Gamma\left(a+3/2\right)}{\sqrt{2}\pi^{3/2}\Gamma\left(2
+a\right)} s^{3/a}(t-s)^2 t^{-2(1+2/a)} g_1\left(t^{1/a},\sqrt{2}\: x\right).
\]
\end{lemma}
\begin{proof}
Apply Lemma \ref{L:gaLowB} 
with $t$ replaced by $s^{1/a}$ and $a=1$ to see that
\begin{multline}\label{E5_:ga2g1}
\left(g_a^2\left(t-s,\sqrt{2}\; \cdot\right) *
g_1\left(s^{1/a},\cdot\right)\right)(x)
\\
\ge
\pi s^{1/a}g_1\left(s^{1/a},\sqrt{2}\: x\right)
\int_\R \ud y\: g_a^2\left(t-s,\sqrt{2} \; y\right)
g_1\left(s^{1/a},\sqrt{2}\: y\right).
\end{multline}
Observe that for $0\le s\le t$,
\[
g_1\left(s^{1/a},\sqrt{2}\: x\right) = \frac{1}{\pi} \: \frac{s^{1/a}}{s^{2/a}
+2x^2}
\ge
\frac{1}{\pi} \: \frac{s^{1/a}}{t^{2/a}
+2x^2}
= \frac{s^{1/a}}{t^{1/a}} g_1\left(t^{1/a},\sqrt{2}\: x\right).
\]
Therefore,
\begin{align*}
\int_\R\ud y\: g_a^2\left(t-s,\sqrt{2}\; y\right)
g_1\left(s^{1/a},\sqrt{2}\: y\right)
& =
\frac{1}{\pi^3}
\int_\R\ud y\:
\frac{(t-s)^2}{\left((t-s)^{2/a}+2 y^2\right)^{1+a}}
\frac{s^{1/a}}{\left(s^{2/a}+2y^2\right)}\\
& \ge
\frac{1}{\pi^3} s^{1/a}(t-s)^2
\int_\R\ud y\:
\frac{1}{\left(t^{2/a}+2y^2\right)^{a+2}}\\
&=
\frac{1}{\sqrt{2}\: \pi^3} s^{1/a} (t-s)^2 t^{-(2+3/a)}
\int_\R
\frac{\ud u}{\left(1+u^2\right)^{a+2}}\:.
\end{align*}
By change of the variable $u=\tan(\theta)$,
\[
\int_\R
\frac{\ud u}{\left(1+u^2\right)^{a+2}} = 2\int_0^{\pi/2}
\cos^{2(a+1)}(\theta) \ud \theta =
\frac{\sqrt{\pi}\:\Gamma\left(a+3/2\right)}{\Gamma\left(2+a\right)},
\]
where the last integral is Euler's Beta integral in the form of
\cite[(5.12.2), p.142]{NIST2010}
\[
\int_0^{\pi/2} \ud \theta\: \sin^{2a-1}(\theta) \cos^{2b-1}(\theta)
= \frac{1}{2}\:\frac{\Gamma(a)\Gamma(b)}{\Gamma(a+b)},\qquad
\Re(a)>0,\:\Re(b)>0\:.
\]
Putting the above two lower bounds back to \eqref{E5_:ga2g1} proves Lemma
\ref{L:ga2g1}.
\end{proof}

\begin{lemma}\label{L:minY}
Suppose $\beta> 1$. For all $x\in\R$,
\[
\min_{y\in\R}  |x-y|^\beta+|y|\ge\left\{
      \begin{array}{ll}
			   \beta^{\frac{\beta}{1-\beta}} + \left||x|-\beta^{\frac{1}{1-\beta}}\right| & \mbox{if } \vert x \vert \geq \beta^{\frac{1}{1-\beta}},\\
				\vert x \vert^{\beta} & \mbox{otherwise.}
			\end{array}
   \right.
\]
\end{lemma}
\begin{proof}
Fix $x\in\R$ and set $f(y)=|x-y|^\beta+|y|$. Assume first that $x \ge 0$. By studying the sign of the derivative of $f'(y)$, we find that if $x \geq \beta^{\frac{1}{1-\beta}}$, then $f$ achieves its
minimum at $y = x - \beta^{\frac{1}{1-\beta}}$. If $0\leq x \leq \beta^{\frac{1}{1-\beta}}$, then $f$ achieves it minimum at $0$. The case $x<0$ is treated similarly.
\end{proof}

%

\begin{proof}[Proof of Theorem \ref{T:Growth}]
(1) In the following, we use $C$ to denote some nonnegative constant, which may depend on $a$, $\delta$ and $\Lip_\rho$,
and can change from line to line.
Fix $p\ge 2$.
By \eqref{E:Jsq0K} and \eqref{E:IntS}, when $t>1$,
\[
\left(J_0^2\star\widehat{\calK}_p\right)(t,x)\le
C A_a t^{2} \left(1+e^{\widehat{\gamma}_p^{a^*}t} \right)  \left|J_0(t,x)\right|,
\]
where the constants $A_a$ and $\widehat{\gamma}_p$ are defined in \eqref{E:Aa} and \eqref{E:D-gamma}, respectively.
By \eqref{E:SecMom-Up} and \eqref{rd3.20}, for $\alpha\ge 0$,
\[
\lim_{t\rightarrow+\infty}\frac{1}{t}\sup_{|x|\ge \exp(\alpha t)} \log \Norm{u(t,x)}_p^2=
\lim_{t\rightarrow+\infty}\frac{1}{t}\sup_{|x|\ge \exp(\alpha t)} \log \left(J_0^2\star\widehat{\calK}_p\right)(t,x) \le
 \widehat{\gamma}_p^{a^*} - \alpha\beta.
\]
Now, $\widehat{\gamma}_p^{a^*} - \alpha\beta<0$ if and only if $\alpha>\beta^{-1}\: \widehat{\gamma}_p^{a^*}$.
Therefore,
\[
 \overline{e}(p) :=
\inf\left\{\alpha>0: \underset{t\rightarrow \infty}{\lim}
\frac{1}{t}\sup_{|x|\ge \exp(\alpha t)} \log \E\left(|u(t,x)|^p\right) <0
\right\}\le \frac{\widehat{\gamma}_p^{a^*}}{\beta}<+\infty.
\]

   Concerning the sufficient condition for \eqref{rd3.20}, suppose that for some $\eta >0$, $\int_\R \vert \mu\vert(dy) (1+ \vert y \vert^{\eta}) < \infty$. We consider first the case where $\eta \in
\,]0,1+a[$. Then by \eqref{E:Gtx-bd},
$$
   \vert J_0(t,x)\vert  \leq \int_\R \vert \mu\vert(dy)\, \frac{K_{a,0}(1+t)}{1+\vert x-y\vert^{1+a}}
	\leq C K_{a,0}(1+t) \sup_{y\in\R} \frac{1}{[(1+\vert y \vert)(1+\vert x-y\vert)^{(1+a)/\eta}]^\eta}.
$$
Let $\tilde \beta = (1+a)/\eta > 1$. Notice that
$$
   (1+|x-y|^{\tilde \beta}) (1+|y|)\ge 1+
   |x-y|^{\tilde \beta} + |y|.
$$
By Lemma \ref{L:minY}, we see that
$$
    \vert J_0(t,x)\vert  \leq \tilde C (1+t) \frac{1}{1+\vert x \vert^{\eta}},
$$
which is condition \eqref{rd3.20} with $\beta = \eta$.

   Now consider the case where $\eta \geq 1+a$. Notice that if $\eta > 1+a$, then we generally do not expect \eqref{rd3.20} to hold with $\beta = \eta$, since for instance, $J_0(t,x) \sim 1/\vert
x\vert^{1+a}$ as $\vert x \vert \to \infty$ when $\mu = \delta_0$. Observe that
$$
	 \vert J_0(t,x)\vert \leq \int_{\R} \frac{\vert \mu\vert(dy)}{1+\vert x-y\vert^{1+a}}\ \lMr{\delta}{G}{a}(t,x-y) \left(1+\vert x-y\vert^{1+a}\right).
$$
From \eqref{E:Gtx-bd}, we deduce that for $t \geq 1$,
$$
    \lMr{\delta}{G}{a}(t,x-y) \left(1+\vert x-y\vert^{1+a}\right) \leq K_{a,n}\, t.
$$
Let $\varphi = \eta /(1+a)$, so that $\varphi \geq 1$. Since for some $\tilde c >0$,
\begin{align*}
   (1+\vert x-y\vert^2)(1+\vert y \vert^{2\varphi}) &\geq \frac{1}{2} + \vert x-y\vert^2 + \frac{1}{2} + \vert y \vert^{2\varphi} \geq (\tilde c \wedge \frac{1}{2})[1+\vert x-y\vert^2 + \vert
y\vert^2]\\
	 & \geq (\tilde c \wedge \frac{1}{2}) \left(1+\frac{x^2}{2}\right),
\end{align*}
we see that for all $t \geq 1$ and $x \in \R$, there is $c>0$ such that
\begin{align*}
   \vert J_0(t,x)\vert & \leq C K_{a,n}\, t \int_{\R} \frac{\vert \mu\vert(dy)}{[(1+\vert x-y\vert^{2})(1+\vert y \vert^{2\varphi})]^{(1+a)/2}}\ (1+\vert y \vert^\eta)\\
	&\leq \tilde C \frac{t}{(1+ x^2)^{(1+a)/2}} \int_{\R} \vert \mu\vert(dy)\, (1+\vert y \vert^\eta),
\end{align*}
which implies \eqref{rd3.20} with $\beta = 1+a$.

(2) We only need to consider the case $p=2$ because $\underline{e}(p) \ge
\underline{e}(2)$ for $p\ge 2$.
Assume first that $\vip=0$. Fix $\epsilon>0$, choose the constant
$\widehat{C}$ according to Lemma \ref{L:J0LowB} such that
\[
J_0(t,x)=\left(\lMr{\delta}{G}{a}(t,\cdot)*\mu\right) \ge
I_{0,l}(t,x):=\widehat{C} \:
1_{[\epsilon,\infty[}(t)\, g_a \left(t,\sqrt{2}\:x\right).
\]
By \eqref{E:SecMom-Lower},
\[
\Norm{u(t,x)}_2^2 \ge J_0^2(t,x) +
\left(J_0^2\star\underline{\calK}\right)(t,x)\ge
\left(I_{0,l}^2\star\underline{\calK}\right)(t,x).
\]
Set $b=2-2/a$ and let $\Upsilon= \Upsilon\left(\lip_\rho,a,\delta\right)$ (see
\eqref{E:Upsilon}).
By Proposition \ref{P:LowBdd-K} and Lemma \ref{L:ga2g1},
\begin{multline*}
\left(I_{0,l}^2\star\underline{\calK}\right)(t,x) \ge
\widehat{C}^2 C \int_0^{t-\epsilon}\ud s \: s^{b-1}\:E_{b,b}\left(\Upsilon
s^b\right)\int_\R\ud y\:
g_a^2\left(t-s,\sqrt{2}y\right)g_1\left(s^{1/a},x-y\right)
\\
\ge
\frac{\widehat{C}^2  C \Gamma\left(a+3/2\right)}{\sqrt{2}\: \pi^{3/2}\Gamma(2+a)}
g_1\left(t^{1/a},\sqrt{2}\:x\right)\: t^{-2(1+2/a)}
\int_0^{t-\epsilon} \ud s\:  s^{b-1}\:E_{b,b}\left(\Upsilon s^b\right)
s^{3/a} (t-s)^2 .
\end{multline*}

Now the integral can be bounded as follows:
\begin{align*}
\int_0^{t-\epsilon} \ud s\:  s^{b-1} \:E_{b,b}\left(\Upsilon s^b\right)
s^{3/a}
(t-s)^2
& \ge
\epsilon^2
\int_0^{t-\epsilon}\ud s\:
E_{b,b}\left(\Upsilon s^b\right) s^{b-1+3/a},
\end{align*}
and
\begin{align*}
\int_0^{t-\epsilon}\ud s\:
E_{b,b}\left(\Upsilon s^b\right) s^{b-1+3/a}
&=
\int_0^{t-\epsilon}\ud s\:
\sum_{n=0}^\infty \frac{\Upsilon^n s^{(n+1)b-1 +3/a}}{\Gamma(bn+b)} \\
&=
\sum_{n=0}^\infty
\frac{\Upsilon^n (t-\epsilon)^{(n+1)b+3/a}}{((n+1)b+3/a)\Gamma((n+1)b)}\;.
\end{align*}
Since $3/a\le 3$, we have
\begin{align*}
\int_0^{t-\epsilon} \ud s\:
E_{b,b}\left(\Upsilon s^b\right) s^{b-1+3/a}
&\ge
(t-\epsilon)^{b+3/a} \sum_{n=0}^\infty
\frac{\Upsilon^n (t-\epsilon)^{b n}}{\Gamma((n+1)b+4)}\\
&=
(t-\epsilon)^{b+3/a}E_{b,b+4}\left(\Upsilon (t-\epsilon)^b \right).
\end{align*}
Therefore, we have
\begin{align}\label{E5_:J0Klow}
 \left(I_{0,l}^2\star\underline{\calK}\right)(t,x) &\ge
\overline{C} \: g_1\left(t^{1/a},\sqrt{2}\:x\right)\: t^{-2(1+2/a)}
\: (t-\epsilon)^{b+3/a} E_{b,b+4}\left(\Upsilon
(t-\epsilon)^b\right),
\end{align}
where
\[
\overline{C} = \frac{\epsilon^2 \: \widehat{C}^2 \:  C\:
\Gamma\left(a+3/2\right)}{\sqrt{2}\: \pi^{3/2}\:\Gamma(2+a)}\;.
\]
Because $x\mapsto g_a(t,x)$ is an even function, decreasing for $x\ge 0$,
we deduce that for all $\beta\ge 0$,
\begin{align*}
\sup_{|x|>\exp(\beta t)} \Norm{u(t,x)}_2^2
\ge
\overline{C} \: g_1\left(t^{1/a},\sqrt{2}\:\exp(\beta t) \right)\: t^{-2(1+2/a)}
 (t-\epsilon)^{3/a-1} E_{b,b+4}\left(\Upsilon
(t-\epsilon)^b\right).
\end{align*}
Because $a\in \:]1,2[\;$, there exists $t_0\ge 0$ such that for all $t\ge t_0$, $t^{2/a} \le 2 e^{2\beta t}$,
so
\begin{align*}
g_1\left(t^{1/a},\sqrt{2}\:\exp(\beta t) \right)
&= \frac{1}{\pi} \frac{t^{1/a}}{t^{2/a}+2e^{2\beta t}}
\ge
\frac{1}{\pi} \frac{t^{1/a}}{4e^{2\beta t}}.
\end{align*}

Finally, by the asymptotic expansion of the Mittag-Leffler function in Lemma
\ref{L:Eab},
\begin{align}\label{E5_:GrowthIn}
\lim_{t\rightarrow\infty}\frac{1}{t}
\sup_{|x|>\exp(\beta t)} \log\Norm{u(t,x)}_2^2
\ge
\Upsilon^{1/b}  -2\beta
\;.
\end{align}
Therefore,
\begin{align*}
\underline{e}(2)&=\sup\left\{\beta>0: \lim_{t\rightarrow+\infty}
\frac{1}{t}\sup_{|x|>\exp(\beta t)} \log \Norm{u(t,x)}_2^2 >0\right\}\\
&\ge \sup\left\{\beta>0: \Upsilon^{1/b} - 2\beta
>0\right\}
= \frac{\Upsilon^{1/b}}{2}\;.
\end{align*}

Now let us consider the case where there is $c>0$ with $J_0 \geq c$, or $\vip\ne 0$. In this case, by \eqref{E:SecMom-Lower} and Proposition
\ref{P:LowBdd-K},
\[
\Norm{u(t,x)}_2^2 \ge (c^2 + \vip^2) \left(1\star \underline{\calK}\right)(t,x)
\ge C t^{b}
E_{b,b+1}\left(\Upsilon t^b\right).
\]
This lower bound does not depend on $x$ and hence, by Lemma \ref{L:Eab},  we get \eqref{E5_:GrowthIn} with the right-hand side replaced by $\Upsilon^{1/b}$.
This completes the proof of Theorem \ref{T:Growth}.
\end{proof}

\begin{proof}[Proof of Theorem \ref{T:Intermittency} (2)]
If $\vip\ne 0$, then from \eqref{E:SecMom-Lower} and Proposition
\ref{P:LowBdd-K}, for some constant $C>0$,
\[
\Norm{u(t,x)}_2^2\ge \vip^2\left(1\star\underline{\calK}\right)(t,x)
\ge
C \vip^2 t^b E_{b,b+1}\left(\Upsilon(\lip_\rho,a,\delta) t^b\right)
\;,
\]
where $b=2-2/a$ and the constant $\Upsilon(\lip_\rho,a,\delta)$ is defined in
\eqref{E:Upsilon}.
Then use the asymptotic expansion of $E_{\alpha,\beta}(z)$ in Lemma \ref{L:Eab} to obtain
\begin{align}
 \label{E:m2}
 \underline{m}_2(x)\ge
\Upsilon\left(\lip_\rho,a,\delta\right)^{1/b}\;.
\end{align}
If $\vip=0$, then from \eqref{E:SecMom-Lower}, \eqref{E5_:J0Klow} and the asymptotics of $E_{\alpha,\beta}(z)$ in Lemma
\ref{L:Eab}, we obtain, via the calculation that led to \eqref{E5_:GrowthIn}, but without replacing $x$ by $\exp(\beta t)$, the same lower bound as \eqref{E:m2}.
Note that this lower bound does not depend on $x$.
This proves the statement (2) with $p=2$. For $p>2$, we use H\"older's inequality
\[
\E\left[|u(t,x)|^2\right] \le \E\left[|u(t,x)|^p\right]^{2/p}\:.
\]
Hence, $\underline{m}_p(x) \ge \frac{p}{2}\; \underline{m}_2(x)$.
This completes the proof of  Theorem \ref{T:Intermittency}.
\end{proof}

\section*{Acknowledgements}
The authors thank Davar Khoshnevisan for some interesting discussions on this problem.
A preliminary version of this work was presented at a conference in Banff, Canada, in April 2012.


\addcontentsline{toc}{section}{Bibliography}
\def\polhk#1{\setbox0=\hbox{#1}{\ooalign{\hidewidth
  \lower1.5ex\hbox{`}\hidewidth\crcr\unhbox0}}} \def\cprime{$'$}

%

\end{document}